\documentclass[aop]{imsart}

\RequirePackage{amsthm,amsmath,amsfonts,amssymb,mathrsfs}
\RequirePackage[numbers]{natbib}

\startlocaldefs

\numberwithin{equation}{section}

\theoremstyle{plain}
\newtheorem{lem}{Lemma}[section]
\newtheorem{thm}[lem]{Theorem}
\newtheorem{prp}[lem]{Proposition}
\newtheorem{cly}[lem]{Corollary}

\theoremstyle{remark}
\newtheorem{rem}{Remark}

\def\Eb{\mathbb{E}}
\def\Nb{\mathbb{N}}
\def\Pb{\mathbb{P}}
\def\Rb{\mathbb{R}}
\def\Tb{\mathbb{T}}
\def\Zb{\mathbb{Z}}
\def\Ac{\mathcal{A}}
\def\Bc{\mathcal{B}}
\def\Cc{\mathcal{C}}
\def\Ec{\mathcal{E}}
\def\Fc{\mathcal{F}}

\def\Ic{\mathcal{I}}
\def\Mc{\mathcal{M}}
\def\Pc{\mathcal{P}}
\def\Qc{\mathcal{Q}}
\def\Rc{\mathcal{R}}

\def\Zc{\mathcal{Z}}
\def\Et{\mathbf{E}}
\def\Pt{\mathbf{P}}
\def\Bs{\mathscr{B}}
\def\Vs{\mathscr{V}}
\def\Ws{\mathscr{W}}
\def\Wh{\Cc_{loc}^{\alpha'}(\Omega)}
\def\Pcf{\Pc^{\text{fin}}}

\def\d{\mathrm{d}}
\def\sva{a}
\def\ind{\mathbf{1}}
\def\ed{\stackrel{\text{\tiny{d}}}{=}}
\def\Var{\mathrm{Var}}

\newcommand\red{}

\endlocaldefs

\begin{document}

\begin{frontmatter}

\title{The two-dimensional continuum random field Ising model}
\runtitle{The two-dimensional continuum random field Ising model}

\begin{aug}
\author[A]{\fnms{Adam} \snm{Bowditch}\ead[label=e1,mark]{a.m.bowditch@gmail.com}}
\and
\author[A]{\fnms{Rongfeng} \snm{Sun}\ead[label=e2,mark]{matsr@nus.edu.sg}}

\address[A]{Department of Mathematics,
		National University of Singapore,
 \printead{e1,e2}}
\end{aug}

\begin{abstract}
In this paper we construct the two-dimensional continuum random field Ising model via scaling limits of a random field perturbation of the critical two-dimensional Ising model with diminishing disorder strength. Furthermore, we show that almost surely with respect to the continuum random field given by a white noise, the law of the magnetisation field is singular with respect to that of the two-dimensional continuum pure Ising model constructed by Camia, Garban and Newman in \cite{cagane15}.
\end{abstract}



\end{frontmatter}
\setcounter{tocdepth}{3}
\tableofcontents

\section{Introduction}
\paragraph*{}

Since its introduction by Lenz \cite{le20} as a model for ferromagnetism, the Ising model has become one of the most fundamental models in statistical mechanics, maintaining an important role in the theory of critical phenomena since Peierls \cite{pe36} proved that it undergoes a phase transition in dimensions two and above. It is natural to consider disorder perturbations of the model by i.i.d.\ random external fields, known as the {\em random field Ising model (RFIM)}, and ask whether the critical behaviour changes or not.
Imry and Ma \cite{imma75} gave a physical argument which suggested that in low dimension, the phase transition is \emph{rounded off} under the influence of arbitrarily weak random field disorder.
This was confirmed in dimension $d=2$ by Aizenman and Wehr \cite{aiwe90} who showed the absence of a first order phase transition at any temperature for any non-zero disorder strength. For dimension $d\geq 3$, the question was settled by Bricmont and Kupiainen \cite{brku88} who showed that the first order phase transition persists at low temperatures.

In this paper we consider the Ising model in dimension $d=2$ with the aim of further understanding the issue of disorder relevance; that is, how the addition of arbitrarily weak disorder changes the nature of the phase transition of the underlying pure model.
We show that disorder relevance manifests itself via the convergence of the disordered model to a disordered continuum limit when the disorder strength and lattice mesh are suitably rescaled. In the absence of disorder, such a continuum limit for the critical two-dimensional Ising magnetisation field has been constructed by Camia, Garban and Newman in \cite{cagane15}. We prove that, for almost every instance of disorder, the pure and disordered continuum limits are singular.

In the particular case of Gaussian disorder, the decay rate of the spin correlations for the RFIM has been of much recent interest.
It has been shown in \cite{ch18} that, at any temperature and any disorder strength, the correlations between spins of distance $N$ are at most $1/ \log \log N$. The decay rate has then been improved to polynomial order in \cite{aipe19}, and then further improved to exponential decay in \cite{dixi19, dixi19a}  (see also \cite{aihape19}), which resolves a long-standing conjecture.

\medskip

Let us first recall the basic ingredients before stating our results.

\paragraph*{The pure Ising model} Let $\Omega\subset \Rb^2$ be a simply connected bounded open domain with a \red{piecewise $C^1$ boundary}. For $\sva>0$, define $\Omega_\sva:=\Omega \cap \sva \Zb^2$ and write $x \sim y$ if $x,y \in \Omega_\sva$ are neighbouring vertices. Denote by $\partial \Omega_\sva:=\{y\in \sva \Zb^2 \backslash \Omega_\sva: y\sim x \mbox{ for some } x\in \Omega_\sva\}$ the external boundary of $\Omega_\sva$. Given boundary condition $\xi\in \{\pm1\}^{\partial \Omega_\sva}$, we then define the \emph{pure Ising model} as the law over spins $\sigma\in\{\pm 1\}^{\Omega_\sva\cup \partial \Omega_\sva}$ with $\sigma|_{\partial \Omega_\sva}=\xi|_{\partial \Omega_\sva}$ by
\begin{equation}
\Pt_{\Omega}^{\sva, \xi}(\sigma) = \frac{1}{Z_{\Omega}^{\sva, \xi}}\exp\left(\beta\sum\limits_{\stackrel{x\sim y}{x \in \Omega_\sva, y\in \Omega_\sva \cup \partial \Omega_\sva}}\sigma_x\sigma_y\right)
\end{equation}
where the sum is over unordered pairs $x\sim y$ and $Z_{\Omega}^{\sva, \xi}$ 
is the \emph{partition function}. When $\xi\equiv+1$, it is known as the $+$ boundary condition and we simply write $+$ in place of $\xi$.
It is well known that there is a critical inverse temperature $\beta_c=\log(1+\sqrt{2})/2$ such that the boundary effect is negligible in the infinite volume limit for $\beta<\beta_c$, leading to a unique infinite volume Gibbs state; and non-negligible for $\beta>\beta_c$, leading to multiple infinite volume Gibbs states, see \cite[Chapter 3]{frve17} for a more detailed introduction. Henceforth, we set $\beta=\beta_c$ and let $\Pt_\Omega^{\sva}$ denote the \emph{two dimensional critical Ising model with $+$ boundary condition}, and we denote by $\Et_\Omega^{\sva}$ expectation with respect to this law. We will assume $+$ boundary condition throughout the rest of the paper and omit $+$ from the superscripts.

\paragraph*{The random field Ising model}
The random field Ising model is a disorder perturbation of the Ising model by introducing i.i.d.\ random external field for each spin.
Denote by $\Pb$ a law over a family $\omega=(\omega_x)_{x\in\Zb^2}$ of i.i.d.\ centred random variables with unit variance and finite exponential moments. Write $\Eb$ for expectation with respect to this law. Given $\lambda, h:\Omega\rightarrow \Rb$, for each $\sva>0$ and $x \in \Omega_\sva$, write
\begin{equation}\label{lambdaxa}
\lambda_x^\sva :=\sva^{7/8}\lambda(x),  \qquad h_x^\sva :=\sva^{15/8}h(x),  \qquad \omega^\sva_x:=\omega_{x/\sva}.
\end{equation}
For $\omega$ fixed, we define the \emph{two dimensional critical random field Ising model} (with $+$ boundary condition) as the law
\begin{equation}\label{Gibbs}
\Pt_{\Omega; \lambda,h}^{\omega,\sva}(\sigma)
:= \frac{1}{Z_{\Omega;\lambda, h}^{\omega,\sva}}\exp\left(\sum_{x\in\Omega_\sva}(\lambda_x^\sva\omega^\sva_x+h_x^\sva)\sigma_x\right)\Pt_{\Omega}^{\sva}(\sigma)
\end{equation}
where \red{$\Pt_{\Omega}^{\sva}$ is the pure Ising model at critical inverse temperature $\beta_c$ and}
\begin{equation}\label{Z}
Z_{\Omega;\lambda, h}^{\omega,\sva}
= \Et_{\Omega}^\sva\left [\exp\left (\sum_{x\in\Omega_\sva}(\lambda_x^\sva\omega^\sva_x+h_x^\sva)\sigma_x\right )\right ]
\end{equation}
is the random partition function depending on the random field $\omega$.

In \cite[Theorem 3.14]{casuzy17}, it is shown that the rescaled partition function
\begin{equation}\label{tildeZ}
\widetilde{Z}_{\Omega;\lambda, h}^{\omega,\sva}:=\theta_\sva Z_{\Omega;\lambda, h}^{\omega,\sva}, \qquad \mbox{where} \quad \theta_\sva:=e^{-\frac{1}{2}\sva^{-1/4}\Vert \lambda\Vert _{L^2}^2},
\end{equation}
converges in $\Pb$-distribution to a non-trivial limit $\Zc_{\Omega;\lambda, h}^W$, which admits a Wiener chaos expansion with respect to a spatial white noise $W$. This is the first step toward the construction of the two-dimensional continuum random field Ising model and the starting point of our paper, which will be explained further in Section \ref{s:Pre}.

\paragraph*{Magnetisation field}
We will study the convergence of the RFIM through its magnetisation field.
For $\sva>0$ and $x\in\Omega_\sva$, let $S_\sva(x):=\{y\in\Omega:\Vert x-y\Vert _{\infty}<\sva/2\}$ be the box centred at $x$ with side length $\sva$. We then define the \red{\emph{rescaled piecewise constant magnetisation field}} as the distribution
\begin{equation}\label{Phiomega}
\Phi_\Omega^\sva := \sva^{-1/8}\sum_{x\in\Omega_\sva}\sigma_x\ind_{S_\sva(x)}.
\end{equation}
Denote by $\mu_\Omega^{\sva}:=\Pt_\Omega^\sva\circ(\Phi_\Omega^\sva)^{-1}$ the law of $\Phi_\Omega^\sva$ without disorder. It has been shown by Camia, Garban and Newman in \cite{cagane15} that, as $\sva\rightarrow 0$, $\mu^\sva_\Omega$ converges weakly to a limiting probability measure $\mu_\Omega$, which can be regarded as the law of the magnetisation field $\Phi_\Omega$ for the continuum two-dimensional critical Ising model. The magnetisation field $\Phi_\omega^\sva$ was regarded as an element of the Sobolev space ${\cal H}^{-3}$, which was subsequently improved to the optimal Besov-H\"older space $\Cc^\alpha_{loc}(\Omega)$ for $\alpha<-1/8$ by Furlan and Mourrat \cite{fumo17}.

Similarly, for each fixed realisation of the random field $\omega$, define
\begin{equation}\label{mu}
\mu^{\omega,\sva}_{\Omega;\lambda,h}:=\Pt_{\Omega;\lambda,h}^{\omega,\sva}\circ(\Phi_\Omega^\sva)^{-1}
\end{equation}
to be the quenched law of the magnetisation field with disorder $\omega$. The main focus of this article is to show that $\mu^{\omega,\sva}_{\Omega;\lambda,h}$ converges weakly in $\Pb$-distribution to a disordered continuum limit $\mu^{W}_{\Omega;\lambda,h}$ (with $+$ boundary conditions) where the disorder is given by a white noise $W$ that arises as
the limit of the random process
\begin{equation} \label{Womega}
W^{\omega,\sva}:=\red{\sva^{-1} \sum_{x\in\Omega_\sva}\omega_x^\sva \ind_{S_\sva(x)}}.
\end{equation}

\paragraph*{Statement}
We will regard the magnetisation field $\Phi_\Omega^\sva$ as an element of the Besov-H\"older space $\Cc^{\alpha}_{loc}(\Omega)$ for $\alpha<-1/8$ defined in \cite{fumo17} (see Section \ref{s:Pre} for more detail). Let $\Mc_1(\Cc^{\alpha}_{loc}(\Omega))$ denote the space of probability measures on $\Cc^{\alpha}_{loc}(\Omega)$ equipped with the topology of weak convergence, so that for every $\omega$, $\mu^{\omega,\sva}_{\Omega;\lambda,h}\in \Mc_1(\Cc^{\alpha}_{loc}(\Omega))$. Denote by $C^1(\Omega)$ the space of bounded, continuously differentiable functions with bounded first derivatives.

Our first result shows that the disordered continuum limit $\mu^{W}_{\Omega;\lambda,h}$ exists and, for almost every realisation of the white noise $W$,  is a probability measure on $\Cc^{\alpha}_{loc}(\Omega)$ and can be interpreted as the law of the continuum RFIM magnetisation field with external field $W$.

\begin{thm}\label{t:JntCnv}
Let $\lambda,h\in C^1(\Omega)$ with $\lambda_{min}:=\inf_{x\in \overline{\Omega}}\lambda(x)>0$ and let $\alpha<-1/8$, $\alpha'<-1$.
As $\sva\rightarrow 0$,
\begin{equation*}
\left(W^{\omega,\sva},\widetilde{Z}^{\omega,\sva}_{\Omega;\lambda,h},\mu^{\omega,\sva}_{\Omega;\lambda,h}\right)\Rightarrow \left(W,\Zc^{W}_{\Omega;\lambda,h},\mu^{W}_{\Omega;\lambda,h}\right)
\end{equation*}
weakly as random variables in $\Wh\times\Rb\times\Mc_1(\Cc^{\alpha}_{loc}(\Omega))$, where $W$ is white noise, and $\Pb$-a.s., $\Zc^{W}_{\Omega;\lambda,h}$ and $\mu^{W}_{\Omega;\lambda,h}$ are determined uniquely by $W$.
\end{thm}
\begin{rem}
With a slight abuse of notation, we will also use $\Pt^{\omega, \sva}_{\Omega; \lambda, h}$ and $\Et^{\omega, \sva}_{\Omega; \lambda, h}$ to denote probability and expectation with respect to $\mu^{\omega, \sva}_{\Omega; \lambda, h}$, and similarly use $\Pt^W_{\Omega; \lambda, h}$ and $\Et^W_{\Omega; \lambda, h}$ for $\mu^W_{\Omega; \lambda, h}$. We will omit $\omega$ $($or $W)$ and $\lambda$ when $\lambda\equiv 0$, and write $\Pt^\sva_\Omega$ and $\Et^\sva_\Omega$ $($or $\Pt_\Omega$ and $\Et_\Omega)$ when $\lambda=h\equiv 0$.
\end{rem}

A natural strategy for proving convergence of the RFIM measure $\mu^{\omega,\sva}_{\Omega;\lambda,h}$ would be to look at the exponential weight in the disordered Gibbs measure \eqref{Gibbs} and define the candidate continuum disordered model by
\begin{equation}\label{RN}
\frac{\d\mu^W_{\Omega;\lambda,h}}{\d\mu_\Omega}(\sigma)\ ``="\  \frac{1}{\Zc^W}\exp\left(\int \big(\sigma_x \lambda(x)W(x)\d x+\sigma_x h(x)\d x\big)\right).
\end{equation}
However, this formula is not well defined because the continuum magnetisation field $\sigma$ is a generalised function and so is $W$, which makes $\sigma W$ ill-defined. In fact, our next result proves that the disordered continuum limit $\mu^{\omega,\sva}_{\Omega;\lambda,h}$ is almost surely singular with respect to the pure continuum limit $\mu_{\Omega}$, which shows that it is hopeless to define the continuum disordered model directly through a Radon-Nikodym density. However, when averaged over the disorder $W$, the limit is absolutely continuous with respect to $\mu_\Omega$.
\begin{thm}\label{t:Sing}
For $\Pb$-a.e.~$W$, the probability measure $\mu^{W}_{\Omega;\lambda,h}$ is singular with respect to $\mu_{\Omega}$. However, the averaged quenched measure $\Eb\mu^{W}_{\Omega;\lambda,h}$ is absolutely continuous with respect to $\mu_{\Omega}$.
 \end{thm}

\paragraph*{Discussion}
One of the first results that identifies a disordered continuum limit for a disordered system is the work by Alberts, Khanin and Quastel,  who showed that for the directed polymer model in dimension $1$, if the disorder strength is sent to zero at a suitable rate as the lattice spacing tends to 0, then the partition functions converge to the solution of the one-dimensional KPZ equation \cite{alkhqu14a}, while the polymer measure converges to a continuum limit called the {\em continuum directed polymer} \cite{alkhqu14b}. Subsequently, Caravenna, Sun and Zygouras made the observation that such disordered continuum limits should exist for more general disorder relevant systems \cite{casuzy17}, where arbitrarily weak disorder perturbation of an underlying pure model changes its behaviour on large scales, and hence tuning the disorder strength down to zero suitably as the lattice spacing tends to zero allows one to construct a continuum limit with non-trivial disorder dependence. They formulated general criteria for the partition functions of a disordered model to have non-trivial limits, which is the first step to construct the disordered continuum model. These criteria were verified for the disordered pinning model, with the continuum disordered pinning model subsequently constructed in \cite{casuzy16}. They were also verified in \cite{casuzy17} for the partition functions of the RFIM with $+$ boundary condition, which provides the starting point of the present paper.

Both the directed polymer model and the disordered pinning model have a time direction, which allows one to use the Markov property
to construct the continuum models directly from the continuum limit of the partition functions. For the RFIM in two dimensions, such an approach is no longer feasible because it would require knowledge of the continuum limit of the partition functions for all domains with all boundary conditions, and it is not even clear how to define general boundary conditions for the continuum RFIM since the magnetisation field is only a distribution. Instead, we will use characteristic functions to characterise the law of the continuum magnetisation field and prove convergence in Theorem \ref{t:JntCnv}.

Heuristically, the singularity in Theorem \ref{t:Sing} can be understood as follows: the disorder splits the system into sub-domains which behave as essentially independent components with a small random external field, tilting the law with respect to the measure with no external field. This occurs on arbitrarily small scales. If the effect of random tilting on each subdomain remains strong enough on smaller and smaller scales, then this gives singularity. In particular, such singularity should arise on any open subdomain, as we will show in the proof. When the disordered law is averaged, the spatial fluctuation is averaged out and the effect of the disorder is smoothed, removing the singularity. The same heuristic applies to the continuum directed polymer model and disordered pinning model. Indeed, the analogue of Theorem \ref{t:Sing} has been proved for both models, where the Markov property plays an essential role. The proof of Theorem \ref{t:Sing} for the continuum RFIM is much more subtle for the reasons described before and will constitute the bulk of the paper.

Our results also extend the work of Camia, Garban and Newman \cite{cagane15, cagane16}, where they constructed and analysed the near-critical scaling limit of the two-dimensional Ising model, which corresponds to the continuum RFIM with a deterministic external field, namely, $\mu^{\omega,\sva}_{\Omega;\lambda,h}$ with $\lambda=0$. It was then shown in \cite{cajine20} that when $h\equiv c\neq 0$, the continuum model has exponential decay of correlations.

\paragraph*{Open problems}
Given recent breakthroughs on the exponential decay of correlations for the RFIM on $\mathbb Z^2$~\cite{dixi19, dixi19a, aihape19}, it would be very interesting to prove the same result for the continuum RFIM and to understand how the magnetisation field depends on the mean $h$ and strength $\lambda$ of the random field. The latter are interesting and challenging questions that are also open for the lattice RFIM.

For the critical Ising model on $\mathbb Z^2$, instead of considering the magnetisation field as in \cite{cagane15, cagane16} and study its scaling limit, one can also study the interfaces between $+$ and $-$ spins, which was shown to converge to the conformal loop ensemble CLE(3) in \cite{beho19}. When a deterministic external field is present in the continuum, the law of the magnetisation field is tilted and the same should hold for the law of the interfaces. Is it possible to characterise the law of interfaces in the continuum RFIM when disorder is present? By Theorem \ref{t:Sing}, we expect it to be singular with respect to the law of CLE(3). However, if we consider the law of a single interface, such as the interface that arises from imposing a Dobrushin boundary condition, then it is conjectured (communicated to us by Christophe Garban) that almost surely in the disorder, the law of the interface is absolutely continuous with respect to the law of an SLE(3).

Another interesting question is to investigate whether one can make sense of \eqref{RN} in the same spirit as in the solution theory for singular SPDEs~\cite{ha14, gip15}, which also had to deal with products of distributions such as $\sigma W$ in \eqref{RN}. And is it possible to construct the measures $\mu_\Omega$ and $\mu_{\Omega; \lambda, h}^W$ as the equilibrium solutions of SPDEs, similar to the stochastic quantisation of $\Phi^4$ theory~\cite{ha14}? Note that the $\Phi^4_3$ measure can formally be seen as a Gibbs change of measure of the Gaussian free field (GFF), similar in spirit to \eqref{RN}, although it is also singular with respect to the reference GFF~\cite{bagu20}. We also remark that in the study of singular SPDEs, results such as Theorem \ref{t:JntCnv} are known as {\em weak universality} (see e.g.~\cite{haqu18}).

\paragraph*{Organisation}
The rest of the paper is organised as follows. In Section \ref{s:Pre} we provide some background and technical results that we will need for the rest of the paper. This includes details on Ising spin correlations, conformal invariance, Besov-H\"older spaces, white noise and Wiener chaos expansions.

In Section \ref{s:FDD} we prove the uniqueness of the limit in Theorem \ref{t:JntCnv}. The characteristic function of the magnetisation field tested against smooth functions can be written as the ratio of partition functions.
This reduces this part of the proof to establishing joint convergence of the white noise approximation and finite families of partition functions.
For this, we use a Lindeberg principle similar to \cite{casuzy17} and show that the limit is a family of Wiener chaos expansions.

\red{In Section \ref{s:Tight} we complete the proof of Theorem \ref{t:JntCnv} by showing tightness, which can be reduced to the case without a random field that was considered in \cite{fumo17}.}

In Section \ref{s:Sing} we prove Theorem \ref{t:Sing} by studying the Radon-Nikodym derivatives of the continuum magnetisation fields conditioned on their average as well as the average of the white noise on disjoint subdomains. These conditioned Radon-Nikodym derivatives are then controlled by discrete approximations and, via suitable fractional moment bounds, are shown to converge to zero as the conditioning is gradually refined.

In Appendix \ref{s:Radon}, we give a sufficient condition for the weak limit of one sequence of probability measures to be absolutely continuous with respect to the weak limit of a second sequence.

In Appendices \ref{s:mom} and \ref{s:pos}, we prove that the RFIM partition functions $\widetilde Z^{\omega, \sva}_{\Omega;\lambda, h}$ have uniformly bounded positive and negative moments of all orders, and hence the same holds for their continuum limit $\Zc^W_{\Omega; \lambda, h}$.

In Appendix \ref{s:BH}, we show that for any $\Phi$ in a Besov-H\"older space $\Cc^\alpha$ with $\alpha>-1$, we can define $\int_B \Phi$
for any subdomain $B$ with a regular enough boundary, so that the total magnetisation of the continuum RFIM on $B$
is well-defined. This is needed in Section \ref{s:Sing}.

\section{Preliminaries}\label{s:Pre}
In this section, we provide some background and list some technical results needed in the proof.
\paragraph*{Spin correlations}
A feature of many two-dimensional critical statistical mechanical models is that the continuum scaling limits are conformally invariant.
In \cite{sm10}, Smirnov established the conformal invariance of fermionic observables in the critical Ising model.
This facilitated rigorously establishing conformally invariant scaling limits including convergence of Ising loops to a CLE \cite{beho19}, convergence of the Ising interface to a chordal SLE \cite{chduhokesm14} and, fundamental to this work, convergence of spin correlations \cite{chhoiz15}. More specifically, we will need the following result.

\begin{lem} \label{l:spin} \red{Let $\Omega\subset \Rb^2$ be a simply connected bounded open domain with a piecewise $C^1$ boundary}.
For the critical Ising model on $\Omega_\sva:=\Omega \cap \sva \Zb^2$ with $+$ boundary condtition, there is a symmetric function
$\phi^+_\Omega:\bigcup_{k=1}^\infty\Omega^k\rightarrow \Rb$ such that for all $n\in\Nb$ and distinct $x_1,...,x_n\in\Omega$,
\begin{flalign}\label{e:ChHoIz}
\sva^{-k/8}\Et_{\Omega}^{\sva,+}\left [\prod_{i=1}^k\sigma_{x_i}\right ] \longrightarrow
\Cc^k\phi^+_\Omega(x_1,...,x_k) \qquad \mbox{as } \sva\downarrow 0,
\end{flalign}
where $\Cc$ is a known constant, $\sigma_x:=\sigma_{x_\sva}$ with $x_{\sva}:=a \lfloor a^{-1}x\rfloor$, and the convergence holds both pointwise and in $L^2(\Omega^n)$.

There exists a symmetric function $f_\Omega: \bigcup_{k=1}^\infty \Omega^k \rightarrow \Rb^+$ continuous everywhere except on the diagonal, such that uniformly in $I\subset \Omega$ with $|I|=k$, $k\in\Nb$, and $\sva\in(0,1]$, we have
	\begin{gather}\label{e:fUpp}
		0\leq \Et_\Omega^{\sva, +}\Bigg[\prod_{x\in I}\sigma_x\Bigg] \leq \sva^{k/8} f_\Omega(I),	 \\
	          \Vert f_\Omega\Vert_{L^2(\Omega^k)}^2\leq C^k(k!)^{1/4}. \label{fOmbd}
	\end{gather}
\end{lem}
\begin{proof}
The pointwise convergence in \eqref{e:ChHoIz} was established in \cite{chhoiz15}, which was then extended to $L^2$ convergence in \cite[Section 8]{casuzy17} by dominated convergence, using \eqref{e:fUpp} (\cite[Lemma 8.1]{casuzy17}) and \eqref{fOmbd}
(\cite[Lemma 8.3]{casuzy17}). \red{The assumption that $\Omega$ has a piecewise $C^1$ boundary is to ensure that $\int_\Omega \frac{{\rm d}x}{d(x, \partial \Omega)^{1/4}}<\infty$, which is used to prove the $L^2$ convergence in \eqref{e:ChHoIz} (see \cite[(8.25)]{casuzy17}).}
\end{proof}

\paragraph*{White noise and chaos expansions}
The partition functions $Z^{\omega,\sva}_{\Omega;\lambda,h}$ defined in \eqref{Z} encode much of the essential information of the system and will be vital for studying the law of the quenched magnetisation field $\mu^{\omega,\sva}_{\Omega;\lambda,h}$ defined in \eqref{mu}. The  partition functions and their scaling limits have been studied in \cite{casuzy17} using polynomial and Wiener chaos expansions which we now review.

Let $\Tb$ be a \red{countable} set (typically $\Omega_\sva$) and $\Pcf(\Tb):=\{I\subseteq\Tb:|I|<\infty\}$. For $I\in\Pcf(\Tb)$ and a vector $u\in \Rb^{\Tb}$ we write $u^I:=\prod_{i\in I}u_i$. Any function $\psi:\Pcf(\Tb)\rightarrow \Rb$ can be used to define a \emph{multi-linear polynomial}
\[\Psi(u)=\sum_{I\in\Pcf(\Tb)}\psi(I)u^I\]
and we call $\psi$ the \emph{kernel} of $\Psi$. Let $\zeta:=(\zeta_i)_{i\in\Tb}$ be a family of independent random variables. We say that a random variable  admits a \emph{polynomial chaos expansion} with respect to $\zeta$ if it can be expressed as $\Psi(\zeta)$ for some multi-linear polynomial $\Psi$. For $l\in\Nb$, we write $\Psi^{\leq l}$ for the chaos expansion with kernel $\psi_{\leq l}$ defined by $\psi_{\leq l}(I)=\psi(I)\ind_{\{|I|\leq l\}}$.

Recall the definition of $\widetilde{Z}_{\Omega;\lambda, h}^{\omega,\sva} $ from \eqref{tildeZ}. Write $\xi^\sva_x:=\lambda_x^\sva\omega^\sva_x+h_x^\sva$ for the random external field; then, using a high temperature expansion,
\begin{flalign}
\widetilde{Z}_{\Omega;\lambda, h}^{\omega,\sva}
& =\theta_\sva \Et_{\Omega}^\sva\left[\exp\left(\sum_{x\in\Omega_\sva}\xi^\sva_x\sigma_x\right)\right] \notag\\
& =\theta_\sva \Et_{\Omega}^\sva\left[\prod_{x\in\Omega_\sva}\left(\cosh(\xi^\sva_x)+\sigma_x\sinh(\xi^\sva_x)\right)\right] \notag\\
& = \theta_\sva\cosh(\xi^\sva_\cdot)^{\Omega_\sva}\sum_{I\subseteq\Omega_\sva}\Et_{\Omega}^\sva\left[\sigma_\cdot^I\right] \tanh(\xi^\sva_\cdot)^I. \label{e:Prefactor}
\end{flalign}
where the prefactor $\theta_\sva\cosh(\xi_\cdot^\sva)^{\Omega_\sva}$ converges to $1$ in probability by a Taylor expansion (see Lemma \ref{l:Theta}). We want to understand the limiting behaviour of the remaining polynomial chaos expansion; for this we first introduce white noise and the Wiener chaos expansion.

The limit of $\widetilde{Z}_{\Omega;\lambda, h}^{\omega,\sva}$ will be expressed as a Wiener chaos expansion with respect to a {\em white noise} $W$ on $\mathbb R^2$, which can be identified with a Gaussian process $W=(W(f))_{f\in L^2(\Rb^2)}$ with mean $\Eb[W(f)]=0$ and covariance $\Eb[W(f)W(g)]=\int f(x)g(x)\d x$.

If $A_1, A_2,\dots$ are disjoint Borel sets with finite Lebesgue measure then the random variables $W(A_i):=W(\ind_{A_i})$ are independent centred Gaussian random variables with variance the Lebesgue measure of $A_i$. Moreover, we have that the relation $W(\bigcup_{i\geq 1}A_i)=\sum_{i\geq 1}W(A_i)$ holds a.s.\ and we write $\int f(x)W(\d x):=W(f)$ even though $W(\cdot)$ is a.s.\ not a signed measure.

Recalling that $S_\sva(x)$ is the square of side length $a$ centred around $x$ we have that
\begin{flalign}\label{e:vartheta}
\vartheta_x^\sva:=\sva^{-1}\int_{S_\sva(x)}W(\d y)
\end{flalign}
is a standard Gaussian for any $\sva>0$ and $x\in\Rb^2$. Furthermore, the sequence of distributions
\begin{flalign}\label{e:Wa}
 W^{\sva}:=\sva\sum_{x\in\Omega_\sva}\vartheta^\sva_x \delta_x=\sum_{x\in\Omega_\sva}W(S_\sva(x)) \delta_x
\end{flalign}
converges a.s.\ to the white noise $W$.


For white noise $W$ on $\Rb^2$ we define the Wiener chaos expansion with kernel $\phi_\Omega^+$ from \eqref{e:ChHoIz} as
\begin{equation}\label{e:Par}
\Zc_{\Omega;\lambda, h}^{W} = 1+\sum_{n=1}^\infty \frac{\Cc^n}{n!} \idotsint_{\Omega^n}\phi_\Omega^+(x_1,...,x_n)\prod_{i=1}^n(\lambda(x_i)W(\d x_i)+h(x_i)\d x_i).
\end{equation}
By \cite[Theorems 3.14 \& 2.3]{casuzy17}, the rescaled partition function $\widetilde{Z}_{\Omega;\lambda, h}^{\omega,\sva}$ converges in $\Pb$-distribution to $\Zc_{\Omega;\lambda, h}^{W}$, and its Wiener chaos expansion is convergent in $L^2$. In this paper, we will need the following stronger result.
\begin{lem} \label{l:mom}
Let $\widetilde{Z}_{\Omega;\lambda, h}^{\omega,\sva}$ be as in Theorem \ref{t:JntCnv}, then  $\widetilde{Z}_{\Omega;\lambda, h}^{\omega,\sva}$ converges in distribution to $\Zc_{\Omega;\lambda, h}^{W}$ as $\sva\downarrow 0$.
Furthermore,
for any $p \geq 0$ (or $p<0$ with $\omega$ satisfying the concentration of measure inequality \eqref{assD2}), $\Eb[(\Zc_{\Omega;\lambda, h}^{W})^p]= \lim_{\sva\downarrow 0} \Eb[(\widetilde{Z}_{\Omega;\lambda, h}^{\omega,\sva})^p] <\infty$. In particular, $\Zc^W_{\Omega; \lambda, h}>0$ $\Pb$-a.s.
%
\end{lem}
\noindent
We only need to show that for any $p\in\Rb$, $\limsup_{\sva\downarrow 0} \Eb[(\widetilde{Z}_{\Omega;\lambda, h}^{\omega,\sva})^p] <\infty$. We will distinguish between $p>0$ and $p<0$, which will be treated in Appendix \ref{s:mom} and \ref{s:pos} respectively.

\paragraph*{Besov-H\"older spaces}
We now describe the Besov-H\"older spaces $\Cc^\alpha$ in which the magnetisation field takes values, following \cite[Section 2]{fumo17}.
First, for $r\in\Nb$, write $C^r$ to denote the space of $r$-times continuously differentiable functions on $\Rb^2$ and, for $f \in C^r$, define the norm
\[\Vert f\Vert_{C^r}:=\sum_{|i|\leq r}\Vert \partial_i f\Vert_\infty.\]
For $\alpha<0$ let $r_\alpha=-\lfloor\alpha\rfloor$ and
\[\Bs^{r_\alpha} := \{f\in C^{r_\alpha} : \Vert f\Vert_{C^r}\leq 1 \text{ and } \text{Supp}(f) \subset B(0, 1)\}.\]
Then, for $f\in C_c^\infty$, denote
\begin{equation}\label{Calpha1}
\Vert f\Vert_{\Cc^\alpha}:=\sup_{\theta\in(0,1]}\sup_{x\in\Rb^2}\sup_{g\in\Bs^{r_\alpha}}\theta^{-\alpha-2}\int f(y)g\left(\frac{y-x}{\theta}\right)\d y.
\end{equation}
The Besov-H\"older space $\Cc^\alpha$ is the completion of $C_c^\infty$ with respect to the norm $\Vert \cdot\Vert_{\Cc^\alpha}$. For every open domain $\Omega\subset \Rb^2$, the local Besov-H\"older space $\Cc^\alpha_{loc}(\Omega)$ is the completion of $C_c^\infty$ with respect to  the family of seminorms $(\Vert \chi\cdot\Vert_{\Cc^\alpha})_{\chi\in C_c^\infty(\Omega)}$.


There is an equivalent characterisation of $\Cc^\alpha$ through \emph{multi-resolution analysis}. A \emph{multi-resolution analysis} of $L^2$ is an increasing sequence $(V_n)_{n\in\Zb}$ of subspaces of $L^2$, together with a function $\phi\in L^2(\Rb^2)$, such that
\begin{enumerate}
\item
$\bigcup_n V_n$ is dense in $L^2$;
\item
$\bigcap_n V_n=\{0\}$;
\item
$f \in V_n$ if and only if $f(2^{-n}(\cdot)) \in V_0$;
\item
$(\phi(\cdot - k))_{k\in \Zb^2}$ is an orthonormal basis of $V_0$.
\end{enumerate}
\red{Such a function $\phi$ is called the \emph{scaling function} of $(V_n)_{n\in \Zb}$.}
Denote by $W_n$ the orthogonal complement of $V_n$ in $V_{n+1}$. Then for any $r\in\Nb$, there exist $\phi, (\psi^{(i)})_{i=1,2,3}$ such that
\begin{enumerate}
\item
$\phi, (\psi^{(i)})_{i=1,2,3}$ all belong to $C^r_c$ with support in $B(0,1)$;
\item
$\phi$ is the scaling function of a multi-resolution analysis $(V_n)_{n\in \Zb}$;
\item
$(\psi^{(i)}(\cdot-k)_{k\in \Zb^2, i=1,2,3})$ is an orthonormal basis of $W_0$;
\item
for each $i=1,2,3$ and $\beta_1, \beta_2\in \Nb_0$ with $\beta_1+\beta_2<r$, we have $\int x_1^{\beta_1} x_2^{\beta_2}\psi^{(i)}(x)\d x=0$.
\end{enumerate}
For any $n \in \Zb$ and $x\in \Rb^2$, let $\phi_{n,x}(y) := 2^n\phi(2^n(y - x))$ and $\psi^{(i)}_{n,x}(y) := 2^n\psi^{(i)}(2^n(y - x))$, and denote
$\Lambda_n = \Zb^2/2^n$. Then $(\phi_{n,x})_{x\in\Lambda_n}$ is an orthonormal basis of $V_n$, while $(\psi^{(i)}_{n,x})_{x\in\Lambda_n, n\in\Zb, i=1,2,3}$ is an orthonormal basis of $L^2$.

Denote by $\Vs_n$ and $\Ws_n$ the orthogonal projections on $V_n$, $W_n$ respectively. For  $f\in L^2(\Rb^2)$, we have
\begin{equation}\label{e:VW}
\Vs_nf=\sum_{x\in\Lambda_n}\langle f,\phi_{n,x}\rangle_{L^2}\phi_{n,x}, \quad \Ws_nf=\sum_{x\in\Lambda_n, i=1,2,3}\langle f,\psi^{(i)}_{n,x}\rangle_{L^2}\psi^{(i)}_{n,x}.
\end{equation}
We then have that, for any $k\in\Zb$,
\begin{equation}\label{e:fL2}
f=\Vs_kf+\sum_{n=k}^\infty\Ws_nf.
\end{equation}
By \cite[Proposition 2.16]{fumo17}, $\Cc^\alpha$ can be equivalently defined as the completion of $C_c^\infty$ with respect to the norm
\begin{equation}
\Vert f\Vert_{\Cc^\alpha}:=\Vert \Vs_0f\Vert_\infty+
\sup_{n\in\Nb}2^{\alpha n}\Vert \Ws_nf\Vert_\infty.
\end{equation}
For an open domain $\Omega\subset \Rb^2$, $\Cc^\alpha_{loc}(\Omega)$ is the completion of $C_c^\infty$ with respect to  the family of seminorms $(\Vert \chi\cdot\Vert_{\Cc^\alpha})_{\chi\in C_c^\infty(\Omega)}$. \red{We note that the discrete piecewise constant magnetisation field $\Phi_\Omega^\sva$ belongs to the space of distributions $\Cc^\alpha_{loc}(\Omega)$ for any $\alpha<0$, and we write $\langle \Phi,f\rangle$ to denote $\Phi$ acting on the function $f$.}

\begin{rem}
Theorem \ref{t:JntCnv} also holds with $\Cc^\alpha_{loc}(\Omega)$ replaced by the Besov spaces $\Bc_{p,q}^{\alpha,loc}(\Omega)$ for any $p,q\in[1,\infty]$, also constructed in \cite{fumo17}. For simplicity, we restrict ourselves to $\Cc^\alpha_{loc}(\Omega)$ which is continuously embedded in $\Bc_{p,q}^{\alpha,loc}(\Omega)$ for any $p,q\in[1,\infty]$.
\end{rem}

\paragraph*{Taylor expansions}
For $\phi\in L^2(\Omega)$, write
\begin{equation}\label{phixa}
\red{\widetilde{\phi}(x)} :=\sva^{-2}\int_{S_\sva(x)}\phi(y)\d y, \qquad \phi_x^\sva  :=\sva^{15/8}\phi(x), \qquad \red{\widetilde{\phi}_x^\sva} :=\sva^{-1/8}\int_{S_\sva(x)}\phi(y)\d y
\end{equation}
for the smoothed version, the scaled version, the smoothed and scaled version, respectively.
\red{To prove convergence of finite dimensional distributions of the magnetisation field, we will consider its Fourier transform, which leads to 
polynomial chaos expansion of the random field partition functions as in \eqref{e:Prefactor} with $\xi_\cdot^\sva$ replaced by $\xi_\cdot^\sva+i\widetilde{\phi}_\cdot^\sva$. We then use a Lindeberg principle to replace $\tanh(\xi_\cdot^\sva+i\widetilde{\phi}_\cdot^\sva)$ with Gaussian random variables. The following estimates give the required control over the moments of these random variables.}
Using a Taylor expansion we have that $\tanh(x)=x-x^3/6+O(x^5)$ for $x$ small. Therefore using the exponential moments of $\omega$ and recalling that $\xi^\sva_x:=\lambda_x^\sva\omega^\sva_x+h_x^\sva$ and their definitions from \eqref{lambdaxa},
we have that
\begin{flalign}
\Eb[\Re(\tanh(\xi_x^\sva+i\widetilde{\phi}_x^\sva))] & = h_x^\sva+O(\sva^{21/8}), \qquad
\Eb[\Re(\tanh(\xi_x^\sva+i\widetilde{\phi}_x^\sva))^2]  = (\lambda_x^\sva)^2+O(\sva^{7/2}),  \notag\\
\Eb[\Im(\tanh(\xi_x^\sva+i\widetilde{\phi}_x^\sva))]  &= \widetilde{\phi}_x^\sva+O(\sva^{29/8}),   \qquad
\Eb[\Im(\tanh(\xi_x^\sva+i\widetilde{\phi}_x^\sva))^2]  = (\widetilde{\phi}_x^\sva)^2+O(\sva^{11/2}),   \label{e:Cor}\\
& \Eb[\Re(\tanh(\xi_x^\sva+i\widetilde{\phi}_x^\sva))\Im(\tanh(\xi_x^\sva+i\widetilde{\phi}_x^\sva))]
= h_x^\sva\widetilde{\phi}_x^\sva+O(\sva^{9/2}),  \notag
\end{flalign}
where $\Re(\cdot)$ and $\Im(\cdot)$ denote the real and imaginary parts respectively.

\section{Uniqueness of the limit}\label{s:FDD}

To prove Theorem \ref{t:JntCnv}, it suffices to show that the laws of $(W^{\omega,\sva},\widetilde{Z}_{\Omega;\lambda, h}^{\omega,\sva},\mu^{\omega,\sva}_{\Omega;\lambda,h})_{\sva\in(0,1]}$ are tight and the limit is unique. In this section, we will assume tightness and prove uniqueness of the limit.


\subsection{Convergence of the finite dimensional distributions}
Note that we can find a countable set of functions $\Lambda\subset C_c^\infty(\Omega)$ (including the function $0$) such that every $\mu\in \Mc_1(\Cc^\alpha_{loc}(\Omega))$ is uniquely determined by its characteristic functions
\begin{equation}\label{e:Qchar}
\hat \mu(\phi) := \int e^{i \langle \phi, \Phi\rangle} \mu({\rm d}\Phi), \qquad \phi \in \Lambda.
\end{equation}
Given tightness, to show that $(W^{\omega,\sva},\widetilde{Z}_{\Omega;\lambda, h}^{\omega,\sva},\mu^{\omega,\sva}_{\Omega;\lambda,h})$ converges to a unique limit as $a\downarrow 0$, it then suffices to show that $(W^{\omega,\sva},\widetilde{Z}_{\Omega;\lambda, h}^{\omega,\sva}, (\hat \mu^{\omega,\sva}_{\Omega;\lambda,h}(\phi))_{\phi\in \Lambda})$ converges to a unique limit, since the limit of $\hat \mu^{\omega,\sva}_{\Omega;\lambda,h}(\phi)$ must be the characteristic function of the limit of $\mu^{\omega,\sva}_{\Omega;\lambda,h}$ by the continuity of $\hat \mu(\phi)$ in $\mu$. To identify the limit of $(\hat \mu^{\omega,\sva}_{\Omega;\lambda,h}(\phi))_{\phi\in\Lambda}$, note that
\begin{align*}
\hat \mu^{\omega,\sva}_{\Omega;\lambda,h}(\phi) & =\Et^{\omega,\sva}_{\Omega;\lambda,h}\left [ \exp\left (i\left <\phi,\Phi^{\sva}_{\Omega}\right >\right )\right]\\
&=\frac{\Et_{\Omega}^{\sva}\left [\exp\left (\sum_{x\in\Omega_\sva}\left (\lambda_x^\sva\omega^\sva_x+h_x^\sva+i\widetilde{\phi}_x^\sva\right )\sigma_x\right )\right ]}{Z_{\Omega;\lambda, h}^{\omega,\sva}}
\;=\;\frac{\widetilde{Z}_{\Omega;\lambda, h+i\widetilde{\phi}}^{\omega,\sva}}{\widetilde{Z}_{\Omega;\lambda, h}^{\omega,\sva}},
\end{align*}
and convergence would follow if we show convergence of $(\widetilde{Z}_{\Omega;\lambda, h+i\widetilde{\phi}}^{\omega,\sva})_{\phi\in \Lambda}$ to a limit $(\Zc^W_{\Omega; \lambda, h+i\phi})_{\phi\in \Lambda}$. The limiting measure $\mu^W_{\Omega; \lambda, h}$ would then be uniquely determined by its characteristic functions
\begin{equation}\label{e:Qchar2}
\hat \mu^W_{\Omega; \lambda, h}(\phi) := \frac{\Zc^W_{\Omega; \lambda, h+ i\phi}}{\Zc^W_{\Omega; \lambda, h}}, \qquad \phi\in \Lambda,
\end{equation}
which is well-defined since $\Zc^W_{\Omega; \lambda, h}>0$ $\Pb$-a.s.\ by Lemma \ref{l:mom}.

Also note that $W$ is uniquely determined by $W_\varphi:=\langle W, \varphi\rangle $, $\varphi\in \Lambda$.
Therefore uniqueness of the limit for $(W^{\omega,\sva},\widetilde{Z}_{\Omega;\lambda, h}^{\omega,\sva},\mu^{\omega,\sva}_{\Omega;\lambda,h})$ would follow from the convergence in distribution of
\begin{equation}
\big( (W^{\omega,\sva}_\varphi)_{\varphi\in\Lambda}, (\widetilde{Z}_{\Omega;\lambda, h+i\widetilde{\phi}}^{\omega,\sva})_{\phi\in \Lambda}\big).
\end{equation}
This is the content of the following proposition.
\begin{prp}\label{p:JntCnv}
Let $\Lambda\subset C_c^\infty(\Omega)$ be as above, then $\big((W^{\omega,\sva}_\varphi)_{\varphi\in\Lambda}, (\widetilde{Z}_{\Omega;\lambda, h+i\widetilde{\phi}}^{\omega,\sva})_{\phi\in \Lambda}\big)$ converges in finite-dimensional distribution to $\big((W_\varphi)_{\varphi\in \Lambda}, (\Zc_{\Omega;\lambda, h+i\phi}^{W})_{\phi\in \Lambda}\big)$ as $\sva\downarrow 0$, where $\Zc_{\Omega;\lambda, h+i\phi}^{W}$ are defined as in \eqref{e:Par}. Furthermore, $\Pb$-a.s., $(\Zc_{\Omega;\lambda, h+i\phi}^{W})_{\phi\in \Lambda}$ are uniquely determined by $W$.
\end{prp}

\begin{proof}
\red{
The proof of the convergence of $(\widetilde{Z}_{\Omega;\lambda, h+i\widetilde{\phi}}^{\omega,\sva})_{\phi\in \Lambda}$ is based on chaos expansions as outlined in Section \ref{s:Pre}. For $\widetilde \phi=0$, this result was proved in \cite{casuzy17}. Here we adapt the proof to handle the case with a complex external field. It requires several approximations (Lemmas \ref{l:trunc}-\ref{l:terms}), the proof of which will be deferred to the next subsection. Along the way, we will also formulate a multivariate version of the Lindeberg principle (Lemma \ref{l:Lind}]) for polynomial chaos expansions.
}

Let $F, G$ be any finite subsets of $C_c^\infty(\Omega)$.
By \eqref{e:Prefactor} we have that
\begin{equation}\label{pre}
\widetilde{Z}_{\Omega;\lambda, h+i\widetilde{\phi}}^{\omega,\sva}=\theta_\sva\cosh(\xi_\cdot^\sva+i\widetilde{\phi}_\cdot^\sva)^{\Omega_\sva}\sum_{I\subseteq\Omega_\sva}\Et_{\Omega}^{\sva}[\sigma^I_\cdot]\tanh(\xi_\cdot^\sva+i\widetilde{\phi}_\cdot^\sva)^I
\end{equation}
where, by Lemma \ref{l:Theta} below, the prefactor $\theta_\sva\cosh(\xi_\cdot^\sva+i\widetilde{\phi}_\cdot^\sva)^{\Omega_\sva}$ converges to $1$ in probability. It remains to prove the joint convergence of $\left((W^{\omega,\sva}_\varphi)_{\varphi\in F}, (\Upsilon_\phi)_{\phi\in G}\right)$ where
\begin{equation}\label{Upsilon}
\Upsilon_\phi:=\sum_{I\subseteq\Omega_\sva}\psi^\sva(I)\left(\frac{\tanh(\xi_\cdot^\sva+i\widetilde{\phi}_\cdot^\sva)}{\Var(\xi_\cdot^\sva)^{1/2}}\right)^I, \qquad \mbox{and} \quad \psi^\sva(I)=\Var(\xi_\cdot^\sva)^{|I|/2}\Et_{\Omega}^{\sva}[\sigma^I_\cdot].
\end{equation}
Note that $\Var(\xi_\cdot^\sva)=(\lambda_\cdot^\sva)^2$. When $\phi=0$, the convergence of $\Upsilon_\phi$ to $\Zc_{\Omega;\lambda, h}^{W}$ was proved in \cite[Theorem 3.14]{casuzy17}
by verifying conditions in \cite[Theorem 2.3]{casuzy17}. We will adapt the proof to the case of complex external fields.

First we show that $\Upsilon_\phi$ can be truncated to order $l\in\Nb$ such that the error is uniform for small $\sva$ and can be made arbitrarily small by choosing $l$ large. More precisely, let $\Upsilon_\phi^{\leq l}$ denote the restriction of the sum  in \eqref{Upsilon} to $I$ with $|I|\leq l$. Then we have
\begin{lem}\label{l:trunc}
For $\phi\in C_c^\infty(\Omega)$,
\begin{align*}
\lim_{l\rightarrow \infty}\lim_{\sva\rightarrow 0^+}\Eb[|\Upsilon_{\phi}-\Upsilon_{\phi}^{\leq l}|^2] &=0.
\end{align*}
\end{lem}
Next, we linearise $\tanh(\xi_\cdot^\sva+i\widetilde{\phi}_\cdot^\sva)$ and approximate $\Upsilon_\phi^{\leq l}$  by
\begin{equation}
\Xi_\phi^{\le l}  :=\sum_{I\subseteq\Omega_\sva, |I|\leq l}\psi^\sva(I)\left(\frac{\xi_\cdot^\sva+i\widetilde{\phi}_\cdot^\sva}{\lambda_\cdot^\sva}\right)^I
\end{equation}
and show that
\begin{lem}\label{l:lin}
For $\phi\in C_c^\infty(\Omega)$,
\begin{align*}
\lim_{l\rightarrow \infty}\lim_{\sva\rightarrow 0^+}\Eb[|\Upsilon_{\phi}^{\leq l}-\Xi_{\phi}^{\leq l}|^2] &=0.
\end{align*}
\end{lem}
We then further approximate $\Xi_\phi^{\le l}$ by
\begin{equation}
\Theta_\phi^{\le l}  :=\sum_{I\subseteq\Omega_\sva, |I|\leq l} a^{-|I|} \psi^\sva(I)\left(\sva \vartheta_\cdot^\sva+\sva\cdot \frac{h_\cdot^\sva+i\widetilde{\phi}_\cdot^\sva}{\lambda_\cdot^\sva}\right)^I,
\end{equation}
where we have replaced $\omega^a_x$ in $\xi^a_x= \lambda^a_x \omega^a_x+h^a_x$ by the normal random variable $\vartheta_x^\sva:=\sva^{-1}\int_{S_\sva(x)}W(\d y)$ defined from the white noise $W$, and $S_\sva(x)$ is the box centred around $x$ with side length $\sva$. Note that $\Theta_\phi$ is a Wiener chaos expansion with respect to $W$, where by Lemma \ref{l:spin}, $\sva^{-n} \psi^\sva(x_1, \ldots, x_n)$ converges pointwise and in $L^2(\Omega^n)$ to the kernel $\Cc^n \phi^+_\Omega(x_1, \ldots, x_n)\prod_{i=1}^n \lambda(x_i)$ appearing in the definition of $\Zc^W_{\Omega; \lambda, h}$  in \eqref{e:Par}.
Using It\^o isometry and the assumption that $\lambda, h \in C^1_b(\Omega)$, $\inf \lambda>0$, and $\phi\in C_c^\infty(\Omega)$, it is then straightforward to check that $\big((W^{\sva}_\varphi)_{\varphi\in F}, (\Theta_\phi^{\leq l})_{\phi\in G}\big)$ converges in $L^2$ to $\big((W_\varphi)_{\varphi\in F}, (\Zc_{\Omega;\lambda,h+i\phi}^{W,\leq l})_{\phi\in G}\big)$ for any $l\in\Nb$, where given a chaos expansion $\Psi$, $\Psi^{\le l}$ denotes its truncation to terms of order at most $l$.


It was proved in \cite[Theorem 3.14]{casuzy17} that the chaos expansion for $\Zc_{\Omega;\lambda,h}^{W}$ converges in $L^2$. It is easily seen that the proof also applies to $\Zc_{\Omega;\lambda,h+i\phi}^{W}$ with $\phi\in C_c^\infty(\Omega)$ (we can dominate $h+i\phi$ by $|h+i\phi|$).  Therefore $\Zc_{\Omega;\lambda,h+i\phi}^{W,\leq l}$ converges in $L^2$ to $\Zc_{\Omega;\lambda,h+i\phi}^{W}$ as $l\to\infty$. To conclude that $\left((W^{\omega,\sva}_\varphi)_{\varphi\in F}, (\Upsilon_\phi)_{\phi\in G}\right)$ converges in distribution to $\big((W_\varphi)_{\varphi\in F}, (\Zc_{\Omega;\lambda,h+i\phi}^{W})_{\phi\in G}\big)$ as $\sva\downarrow 0$, we can just truncate the chaos expansions to an arbitrarily large order $l$ and then show that for any bounded $f$ with bounded first derivatives, we have
\begin{equation}\label{fcomp}
\lim_{\sva\downarrow 0}|\Eb[f((W^{\omega,\sva}_\varphi)_{\varphi\in F}, (\Xi^{\leq l}_\phi)_{\phi\in G})]-\Eb[f((W^{\sva}_\varphi)_{\varphi\in F}, (\Theta_\phi^{\leq l})_{\phi\in G})]|=0.
\end{equation}
Separating the real and imaginary parts and writing $\Psi_{\phi,\Re}^{\leq l}=\Re(\Psi_\phi^{\leq l})$ and $\Psi_{\phi,\Im}^{\leq l}=\Im(\Psi_\phi^{\leq l})$ for $\Psi\in\{\Upsilon, \Xi,\Theta\}$, it suffices to show that
\begin{lem}\label{l:terms}
For all $g$ bounded and differentiable with bounded first derivatives, we have
\begin{align*}
\lim_{\sva\downarrow 0} \big|\Eb[g((W^{\omega,\sva}_\varphi)_{\varphi\in F}, (\Xi^{\le l}_{\phi,\Re})_{\phi\in G},(\Xi^{\le l}_{\phi,\Im})_{\phi\in G})]-\Eb[g((W^{\sva}_\varphi)_{\varphi\in F}, (\Theta_{\phi,\Re}^{\leq l})_{\phi\in G},(\Theta_{\phi,\Im}^{\leq l})_{\phi\in G})] \big|
\end{align*}
is equal to $0$.
\end{lem}
We will prove Lemmas \ref{l:trunc}--\ref{l:terms} in the next subsection. Since $(\Zc_{\Omega;\lambda, h+i\phi}^{W})_{\phi\in \Lambda}$ is a countable family defined via Wiener chaos expansions with respect to the white noise $W$, they are almost surely determined by $W$. This concludes the proof of Proposition \ref{p:JntCnv}.
\end{proof}

\subsection{Proof of Lemmas \ref{l:trunc}--\ref{l:terms}}

Before starting the proof of Lemma \ref{l:trunc}, we first show that the prefactor in \eqref{pre} tends to $1$ in probability as $\sva\downarrow 0$. The proof is similar to that in the proof of Theorem 3.14 in \cite{casuzy17}.

\begin{lem}\label{l:Theta}
Let $\theta_\sva:=e^{-\frac{1}{2}\sva^{-1/4}\Vert \lambda\Vert _{L^2}^2}$. \red{Then uniformly in $\sigma\in \{\pm1\}^{\Omega_a}$,} we have that
\[
\theta_\sva\Eb\left[\exp\left(\sum_{x\in\Omega_\sva}\lambda_x^\sva\omega_x^\sva\sigma_x\right)\right] \longrightarrow 1 \quad \mbox{as } \sva\downarrow 0.
\]
Moreover, $\theta_\sva \prod_{x\in \Omega_\sva} \cosh(\lambda_x^\sva\omega_x^\sva+h_x^\sva+i\widetilde{\phi}_x^\sva)$ converges to $1$ in $\Pb$-probability for any $\phi\in C_c^\infty(\Omega)$; this convergence also holds in $L^p$ for any $p>0$ when $\phi\equiv 0$.
\end{lem}
\begin{proof}
Recall that $\lambda_x^\sva :=\sva^{7/8}\lambda(x)$ and $h_x^\sva :=\sva^{15/8}h(x)$. \red{Note that
\begin{align*}
\log \Eb\left[\exp\left(\sum_{x\in\Omega_\sva}\lambda_x^\sva\omega_x^\sva\sigma_x\right)\right] = \sum_{x\in\Omega_\sva} \log \Eb\big[e^{\lambda_x^\sva\omega_x^\sva\sigma_x}\big] = \sum_{x\in\Omega_\sva}\Big( \frac{(\lambda^\sva_x)^2}{2} + R(\lambda_x^\sva\omega_x^\sva\sigma_x)\Big),
\end{align*}
where we applied Taylor expansion to the log-moment generating function of $\omega^\sva_x$, with error term
$$
|R(\lambda_x^\sva\omega_x^\sva\sigma_x)| \leq C \Vert \lambda\Vert_\infty \sva^{21/8}.
$$
Therefore
\begin{align*}
&\Bigg|\log \theta_a + \log \Eb\left[\exp\left(\sum_{x\in\Omega_\sva}\lambda_x^\sva\omega_x^\sva\sigma_x\right)\right]\Bigg|\\
\leq\, & \Big| -\frac{1}{2}\sva^{-1/4}\Vert \lambda\Vert _{L^2}^2 + \frac{1}{2} \sum_{x\in\Omega_\sva} (\lambda^\sva_x)^2 \Big| + C\Vert \lambda\Vert_\infty  \sum_{x\in\Omega_\sva} \sva^{21/8}.
\end{align*}
Since $\lambda\in C^1(\Omega)$, the first difference is of order $O(\sva^{7/4})$ and tends to $0$ as $\sva \downarrow 0$ by a Riemann sum approximation with control on the error. The second term also tends to 0 uniformly in $\sigma$, which completes the proof of the first statement.
}

Recall that $\xi_x^\sva=\lambda_x^\sva\omega_x^\sva+h_x^\sva$ and note that by a Taylor expansion, we have that
\[\Eb[\log(\cosh(\xi_x^\sva+i\widetilde{\phi}_x^\sva))]=\frac{(\lambda_x^\sva)^2}{2}+O((h_x^\sva)^2+(\widetilde{\phi}_x^\sva)^2+(\lambda_x^\sva)^4)=\frac{(\lambda_x^\sva)^2}{2}+O(\sva^{7/2})\]
where the error term $O(\sva^{7/2})$ is uniform over $x$ by continuity of $\lambda$, $h$ and $\phi$. \red{In particular,
\[
\sum_{x\in\Omega_\sva}\Eb[\log(\cosh(\xi_x^\sva+i\widetilde{\phi}_x^\sva))] = \sum_{x\in\Omega_\sva} \frac{(\lambda_x^\sva)^2}{2} + O(\sva^{3/2})
= -\log \theta_\sva +O(\sva^{3/2}).
\]
Therefore
\begin{align}\label{e:PreFac}
&\theta_\sva \!\!\!\prod_{x\in \Omega_\sva}\!\!\!\cosh(\xi_x^\sva+i\widetilde{\phi}_x^\sva)\\
=\, &e^{O(\sva^{3/2})} \exp\left(\sum_{x\in\Omega_\sva}\log(\cosh(\xi_x^\sva+i\widetilde{\phi}_x^\sva))-\Eb[\log(\cosh(\xi_x^\sva+i\widetilde{\phi}_x^\sva))]\right).\notag
\end{align}
}
The sum is over $|\Omega_\sva|$ independent centred random variables, each with variance at most
\[\Eb[\log(\cosh(\xi_x^\sva+i\widetilde{\phi}_x^\sva))^2]=O((\lambda_x^\sva)^4)=O(\sva^{7/2}).\]
Therefore \eqref{e:PreFac} converges to 1 in probability by a weak law of large numbers.

Lastly, to show that $\theta_\sva \prod_{x\in \Omega_\sva} \cosh(\lambda_x^\sva\omega_x^\sva+h_x^\sva)$ converges to $1$ in $L^p$ for any $p>0$, it suffices to show that for any $k\in\Nb$, $\Eb[\theta_a^k \prod_{x\in \Omega_\sva} \cosh(\lambda_x^\sva\omega_x^\sva+h_x^\sva)^k]$ is bounded as $\sva\downarrow 0$. This is straightforward to verify by expanding
$(\cosh y)^k = (\frac{e^y +e^{-y}}{2})^k$ and applying Taylor expansion to the exponential moment generating function of $\omega^\sva_x$. We omit the details. We remark that the $L^p$ convergence statement should also hold for $\phi\neq 0$, but the details will be more tedious and is not needed so we left it out.
\end{proof}

\begin{proof}[Proof of Lemma \ref{l:trunc}]
The proof follows by using a Taylor expansion of $\tanh(x)$ to bound moments of $\tanh(\xi_x+i\widetilde{\phi}_\sva)$ and $L^2$ estimates of the kernel $\psi^\sva$. This is similar to \cite[Theorem 2.8]{casuzy17}.
We will bound the difference of the real parts $|\Upsilon_{\phi, \Re}-\Upsilon_{\phi, \Re}^{\leq l}|$. The imaginary part follows similarly.

Let
$$
\Rc_x  :=\Re(\tanh(\xi_x^\sva+i\widetilde{\phi}_x^\sva))/\lambda_x^\sva \qquad \mbox{and} \qquad \Ic_x  :=\Im(\tanh(\xi_x^\sva+i\widetilde{\phi}_x^\sva))/\lambda_x^\sva
$$
and $\widetilde{\Rc}_x :=\Rc_x-\Eb[\Rc_x]$, $\widetilde{\Ic}_x :=\Ic_x-\Eb[\Ic_x]$, so that
$$
\frac{\tanh(\xi_x^\sva+i\widetilde{\phi}_x^\sva))}{\lambda_x^\sva} = \widetilde \Rc_x+ \Eb[\Rc_x] + i\widetilde\Ic_x + i\Eb[\Ic_x].
$$
Let $\psi^\sva_{>l}(I)=\psi^\sva(I)\ind_{|I|>l}$. Expanding \eqref{Upsilon}, we note that
\begin{align}
\Upsilon_\Re^{>l}  := \Upsilon_{\phi, \Re}-\Upsilon_{\phi, \Re}^{\leq l}
& = \!\!\!\!\!\!\!\!\!\!\!\!\!\!\!\!\!\!\!\!  \sum_{\genfrac{}{}{0pt}{3}{K_1, K_2, K_3, K_4\subset \Omega_\sva \mbox{ \footnotesize \red{all disjoint}} }{ |K_2|+|K_4| \mbox{ \footnotesize even}}}
\!\!\!\!\!\!\!\!\!\!\!\!\!\!\!\!\!\!\!\!\!\!\!\!\! (-1)^{(|K_2|+|K_4|)/2}
(\widetilde{\Rc}_\cdot)^{K_1} (\widetilde{\Ic}_\cdot)^{K_2} \Eb[\Rc_\cdot]^{K_3}   \Eb[\Ic_\cdot]^{K_4} \psi^\sva_{>l}(\cup_i K_i) \notag\\
& = \sum_{K_1, K_2 \subset \Omega_\sva \mbox{ \footnotesize disjoint}} (\widetilde{\Rc}_\cdot)^{K_1} (\widetilde{\Ic}_\cdot)^{K_2}
\psi^\sva_{>l}(K_1,K_2), \label{Up>l}
\end{align}
where
\begin{align} \label{psipsi}
\psi^\sva_{>l}(K_1,K_2)
&:= \!\!\!\!\!\!\!\!\!\!\!\!\!  \sum_{\genfrac{}{}{0pt}{3}{K_3, K_4 \subset \Omega_\sva: K_1, K_2, K_3, K_4\mbox{\footnotesize \ \red{all disjoint}} }{ |K_2|+|K_4| \mbox{\footnotesize \  even}}}
\!\!\!\!\!\!\!\!\!\! (-1)^{(|K_2|+|K_4|)/2} \Eb[\Rc_\cdot]^{K_3}\Eb[\Ic_\cdot]^{K_4}  \psi^\sva_{>l}(\cup_i K_i)\notag  \\
&\leq  \!\!\!\!\!\!\!\!\!\!\!\!\!\!  \sum_{\genfrac{}{}{0pt}{3}{K_3 \cap K_4 =\emptyset }{ (K_3\cup K_4)\cap (K_1\cup K_2)=\emptyset}} \!\!\!\!\!\!\!\!\!\!\!\!\!  \psi^\sva_{>l}(\cup_i K_i) (C\sva)^{|K_3\cup K_4|},
\end{align}
and we used the fact that by \eqref{e:Cor},  $|\Eb[\Re_x]|,|\Eb[\Ic_x]|$ are uniformly bounded by $C\sva$ for some $C>0$. Note that the above bound depends only on $K_1\cup K_2$.
Using that $\Eb[(\widetilde{\Rc}_\cdot)^{K_1} (\widetilde{\Ic}_\cdot)^{K_2}(\widetilde{\Rc}_\cdot)^{K'_1} (\widetilde{\Ic}_\cdot)^{K'_2}]=0$ when $K_1\cup K_2 \neq K_1'\cup K_2'$, we can compute
\begin{equation}\label{Up2}
\Eb[(\Upsilon_\Re^{>l})^2] = \sum_{U\subset \Omega_a} \sum_{\genfrac{}{}{0pt}{3}{K_1\cap K_2=K'_1\cap K'_2=\emptyset}{K_1\cup K_2=K_1'\cup K_2'=U}} \!\!\!\!\!\!
\psi^\sva_{>l}(K_1,K_2) \psi^\sva_{>l}(K'_1,K'_2) \Eb[(\widetilde{\Rc}_\cdot)^{K_1} (\widetilde{\Ic}_\cdot)^{K_2}(\widetilde{\Rc}_\cdot)^{K'_1} (\widetilde{\Ic}_\cdot)^{K'_2}].
\end{equation}
The expectation can be bounded using Cauchy-Schwarz and independence as follows:
\begin{equation}\label{Ebound}
\begin{aligned}
\Eb[(\widetilde{\Rc}_\cdot)^{K_1} (\widetilde{\Ic}_\cdot)^{K_2}(\widetilde{\Rc}_\cdot)^{K'_1} (\widetilde{\Ic}_\cdot)^{K'_2}]
& \leq \Eb[(\widetilde{\Rc}^2_\cdot)^{K_1} (\widetilde{\Ic}^2_\cdot)^{K_2}]^{1/2} \Eb[(\widetilde{\Rc^2}_\cdot)^{K'_1} (\widetilde{\Ic^2}_\cdot)^{K'_2}]^{1/2} \\
& = (\Eb[\widetilde{\Rc}^2_\cdot]^{K_1} \Eb[\widetilde{\Rc}^2_\cdot]^{K'_1} \Eb[\widetilde{\Ic}^2_\cdot]^{K_2} \Eb[\widetilde{\Ic}^2_\cdot]^{K_2'})^{1/2} \\
&  \leq (1+C \sva^{7/4})^{|K_1|+|K'_1|} (C a)^{|K_2|+|K'_2|},
\end{aligned}
\end{equation}
where we used that by \eqref{e:Cor},  $\Eb[\widetilde \Rc_x^2]\leq 1+C a^{7/4}$ and $\Eb[\widetilde \Ic_x^2]\leq Ca^2$ for some $C>0$ uniformly in $x\in \Omega_a$. On the other hand, denoting $U:=K_1\cup K_2=K_1'\cup K_2'$, by \eqref{psipsi} we have
\begin{equation}
\begin{aligned}
\psi^\sva_{>l}(K_1,K_2) \psi^\sva_{>l}(K'_1,K'_2) & \leq  \Big(\!\!\!\! \sum_{\genfrac{}{}{0pt}{3}{K_3 \cap K_4 =\emptyset }{(K_3\cup K_4)\cap U=\emptyset}} \!\!\!\!\!\!\! \psi^\sva_{>l}(\cup_i K_i) (Ca)^{|K_3\cup K_4|}\Big)^2  \\
& \leq  \!\!\!\! \sum_{\genfrac{}{}{0pt}{3}{K_3 \cap K_4 =\emptyset }{ (K_3\cup K_4)\cap U=\emptyset}} \!\!\!\!\!\!\! \varepsilon^{|K_3\cup K_4|} \psi^\sva_{>l}(\cup_i K_i)^2  \!\!\!\! \sum_{\genfrac{}{}{0pt}{3}{K_3 \cap K_4 =\emptyset }{ (K_3\cup K_4)\cap U=\emptyset}} \!\!\!\!\!\!\!  (\varepsilon^{-1}C^2a^2)^{|K_3\cup K_4|} \\
& = \sum_{V\subset \Omega_a \backslash U} (2\varepsilon)^{|V|} \psi^\sva_{>l}(U \cup V)^2 \sum_{V\subset \Omega_a \backslash U}
(2\varepsilon^{-1}C^2a^2)^{|V|} \\
& = C_\varepsilon \sum_{V\subset \Omega_a \backslash U} (2\varepsilon)^{|V|} \psi^\sva_{>l}(U \cup V)^2
\end{aligned}
\end{equation}
where $C_\varepsilon$ is a uniform bound on $(1+ 2\varepsilon^{-1}C^2a^2)^{|\Omega_a|}$ for $\sva$ small. Substituting this bound together with \eqref{Ebound} into \eqref{Up2} then gives
\begin{align*}
&\Eb[(\Upsilon_\Re^{>l})^2] \\
 \leq\, & \sum_{U\subset \Omega_a} \sum_{\genfrac{}{}{0pt}{3}{K_1\cap K_2=K'_1\cap K'_2=\emptyset }{ K_1\cup K_2=K_1'\cup K_2'=U}} \!\!\!\!\!\! C_\varepsilon (1+C \sva^{7/4})^{|K_1|+|K'_1|} (C a)^{|K_2|+|K'_2|} \sum_{V\subset \Omega_a \backslash U} (2\varepsilon)^{|V|} \psi^\sva_{>l}(U \cup V)^2 \\
\leq\, & C_\varepsilon \sum_{U\cap V=\emptyset}  (1+Ca)^{|U|} (2\varepsilon)^{|V|}  \psi^\sva_{>l}(U \cup V)^2 \\
=\, & C_\varepsilon \sum_{I\subset \Omega_a}  (1+\eta)^{|I|} \psi^\sva_{>l}(I)^2,
\end{align*}
where $\eta:=Ca+2\varepsilon$ can be made arbitrarily small by choosing $\varepsilon$ small and let $\sva\downarrow 0$. This was shown to converge to $0$ as $\sva\rightarrow 0$ then $l\rightarrow \infty$ in the proof of Theorem 3.14 in \cite[Section 8]{casuzy17}.
\end{proof}

\begin{proof}[Proof of Lemma \ref{l:lin}]
This follows from the moment bounds on $\tanh(\xi_x^\sva+i\widetilde{\phi}_x^\sva)$ used in Lemma \ref{l:trunc} and a suitable representation of the kernel $\psi^\sva$ on sets of size at most $l$. This is similar to \cite[Theorem 2.8]{casuzy17}.
We bound the difference of the real parts $\Upsilon_{\phi, \Re}^{\leq l}-\Xi_{\phi, \Re}^{\leq l}$. The imaginary part follows similarly.

Let $\bar{\Rc}_x:=\xi_x^\sva/\lambda_x^\sva$, $\bar{\Ic}_x:=\widetilde{\phi}_x^\sva/\lambda_x^\sva$, $\widehat{\Rc}_x:=\bar{\Rc}_x-\Eb[\bar{\Rc}_x]$, $\widehat{\Ic}_x:=\bar{\Ic}_x-\Eb[\bar{\Ic}_x]$. Let $\widehat\psi^\sva_{\le l}(K_1,K_2)$ be defined the same as $\psi^\sva_{>l}(K_1,K_2)$ in \eqref{psipsi}, except that $\Eb[\Rc_x]$ and $\Eb[\Ic_x]$ are replaced by $\Eb[\bar{\Rc}_x]$ and $\Eb[\bar\Ic_x]$ respectively, and $\psi_{>l}^\sva(I)$ is replaced by $\psi_{\le l}^\sva(I):=\psi^\sva(I)1_{|I|\le l}$. By the same decomposition as in \eqref{Up>l}, we have
\[
\Xi^{\leq l}_{\phi, \Re} =  \sum_{K_1, K_2 \subset \Omega_\sva \mbox{ \footnotesize disjoint}} (\widehat{\Rc}_\cdot)^{K_1} (\widehat{\Ic}_\cdot)^{K_2}  \widehat \psi^\sva_{\le l}(K_1,K_2).
\]
Therefore
\begin{align}
\frac{1}{2}\Eb[|\Upsilon_{\phi, \Re}^{\leq l}-\Xi_{\phi, \Re}^{\leq l}|^2]
& \leq \Eb\Big[\Big|\sum_{K_1, K_2\subset \Omega_\sva} \big((\widetilde{\Rc}_\cdot)^{K_1}(\widetilde{\Ic}_\cdot)^{K_2}-(\widehat{\Rc}_\cdot)^{K_1}(\widehat{\Ic}_\cdot)^{K_2}\big)\psi^\sva_{\leq l}(K_1,K_2)\Big|^2\Big] \notag\\
& \quad +\Eb\Big[\Big|\sum_{K_1, K_2\subset \Omega_\sva} (\widehat{\Rc}_\cdot)^{K_1}(\widehat{\Ic}_\cdot)^{K_2}(\psi^\sva_{\leq l}(K_1,K_2)-\widehat\psi^\sva_{\leq l}(K_1,K_2))\Big|^2\Big]. \label{e:Spl}
\end{align}
The first term can be bounded in the same way as in the proof of Lemma \ref{l:trunc}. By Taylor expanding $\tanh(\xi^a_x+ i \widetilde{\phi}^a_x)$, we note that, uniformly in $K_1, K_2\subset \Omega_a$ with $|K_1|+|K_2|\le l$, we have
$$
\Eb\Big[\Big(\big(\widetilde{\Rc}_\cdot)^{K_1}(\widetilde{\Ic}_\cdot)^{K_2}-(\widehat{\Rc}_\cdot)^{K_1}(\widehat{\Ic}_\cdot)^{K_2}\Big)^2\Big] = \red{o(1)} \Eb\big[\widetilde{\Rc}_\cdot^2]^{K_1}\Eb[\widetilde{\Ic}^2_\cdot]^{K_2} \qquad \mbox{as } a\downarrow 0.
$$
It is then easily seen that the first term in \eqref{e:Spl} tends to $0$.

The second term in \eqref{e:Spl} can be bounded similarly, using that when we compare the definitions of $\psi^\sva_{\leq l}$ with that of $\widehat \psi^\sva_{\leq l}$ as in \eqref{psipsi}, we note that uniformly in $K_3, K_4\subset \Omega_a$ with $|K_3|+|K_4|\le l$, we have
$$
\Big|\Eb[\Rc_\cdot]^{K_3}\Eb[\Ic_\cdot]^{K_4}- \Eb[\bar{\Rc}_\cdot]^{K_3}\Eb[\bar{\Ic}_\cdot]^{K_4} \Big| = \red{o(1)} \big|\Eb[\Rc_\cdot]^{K_3}\Eb[\Ic_\cdot]^{K_4} \big|  \qquad \mbox{as } a\downarrow 0,
$$
from which it follows that the second term of \eqref{e:Spl} also tends to $0$.
\end{proof}

\begin{proof}[Proof of Lemma \ref{l:terms}] Note that $(W^{\omega,\sva}_\varphi)_{\varphi\in F}, (\Xi^{\le l}_{\phi,\Re})_{\phi\in G}, (\Xi^{\le l}_{\phi,\Im})_{\phi\in G}$ are polynomial chaos expansions with respect to the random variables $(\omega^\sva_x)_{x\in \Omega_\sva}$, with kernels
\begin{equation}\label{kernels}
\begin{aligned}
\psi^\sva_\varphi(I)
& = \begin{cases} \sva\varphi(x) & \text{if } I=\{x\}, \\ 0 & \text{otherwise,} \end{cases} \\
\psi^\sva_{\phi,\Re}(I)
& = \sum_{\genfrac{}{}{0pt}{3}{K_1, K_1 \subset \Omega_\sva \backslash I }{K_1\cap K_2=\emptyset, |K_2|\in 2\Zb}}
\psi_{\leq l}^\sva(I\cup K_1\cup K_2) \Big(\frac{h^\sva_\cdot}{\lambda^\sva_\cdot}\Big)^{K_1} \Big(\frac{\widetilde{\phi}^\sva_\cdot}{\lambda^\sva_\cdot}\Big)^{K_2}(-1)^{|K_2|/2}, \\
\psi^\sva_{\phi,\Im}(I)
& = \sum_{\genfrac{}{}{0pt}{3}{K_1, K_1 \subset \Omega_\sva \backslash I }{ K_1\cap K_2=\emptyset, |K_2|+1\in 2\Zb}}\psi_{\leq l}^\sva(I\cup K_1\cup K_2) \Big(\frac{h^\sva_\cdot}{\lambda^\sva_\cdot}\Big)^{K_1}\Big(\frac{\widetilde{\phi}^\sva_\cdot}{\lambda^\sva_\cdot}\Big)^{K_2}(-1)^{(|K_2|-1)/2},
\end{aligned}
\end{equation}
while $(W^{\sva}_\varphi)_{\varphi\in F}, (\Theta_{\phi,\Re}^{\leq l})_{\phi\in G}, (\Theta_{\phi,\Im}^{\leq l})_{\phi\in G}$ are polynomial chaos expansions with respect to the Gaussian random variables $(\vartheta_x^\sva)_{a\in \Omega_\sva}$ with the same kernel. Lemma \ref{l:terms} then follows from a  Lindeberg principle for polynomial chaos expansions. When only one polynomial chaos expansion is under consideration, such a Lindeberg principle was proved in \cite[Theorems 2.6]{casuzy17}. We need here a multivariate version, which we formulate as follows.

\begin{lem}\label{l:Lind}
Let $\Tb$ be a finite index set, and let $(\omega_x)_{x\in \Tb}$ and $(\vartheta_x)_{x\in \Tb}$ be two different families of i.i.d.\ random variables with zero mean, unit variance, and finite third absolute moments. Let $(\psi_i)_{1\leq i\leq n}$ be a family of kernels for polynomial chaos expansions of order at most $l\in\Nb$, and let $\Psi(\omega):=(\Psi_i(\omega))_{1\leq i\leq n}$ and $\Psi(\vartheta):=(\Psi_i(\vartheta))_{1\leq i\leq n}$ be the corresponding polynomial chaos expansions with respect to $\omega_\cdot$ and $\vartheta_\cdot$, with kernel $\psi_i$. For all $g: \Rb^n \to \Rb$ three times differentiable with $\sup_{|\alpha|\leq 3}\Vert D^\alpha g\Vert _{\infty}<\infty$, there exists a constant $C_{g, l, n}$ such that
\begin{equation}\label{Lind}
|\Eb[g(\Psi(\omega))]-\Eb[g(\Psi(\vartheta))]|
 \leq C_{g,l,n} M^l \sum_{i=1}^{n}  {\rm {\mathbb V}ar}(\Psi_i) \big( \max_{x\in \Tb} {\rm Inf}_x[\psi_i]\big)^{1/2},
\end{equation}
where $M=\max\{\Eb[|\omega_x|^3], \Eb[|\vartheta_x|^3]\}$ and ${\rm Inf}_x[\psi_i] = \sum_{I\ni x} \psi_i(I)^2$.
\end{lem}

We defer the proof of Lemma \ref{l:Lind} and first apply it by verifying that for each of the kernels $\psi^\sva_*$ in \eqref{kernels}, we have that the variance stays bounded while
\begin{equation}\label{e:LinDer}
\lim_{a\downarrow 0} \max_{x\in \Omega_a} {\rm Inf}_x(\psi^\sva_*)=0.
\end{equation}

Note ${\mathbb V}{\rm ar}(W^{\omega, \sva}_\varphi)=\sum_{x\in \Omega_\sva} \sva^{-2} (\int_{S_\sva(x)} \varphi(x) {\rm d}x)^2\to \Vert \varphi\Vert_2^2$, while
${\rm Inf}_x[\psi^\sva_\varphi]=a^2 \varphi(x)^2\leq a^2 |\varphi|_\infty^2$ which tends to $0$ as $a\downarrow 0$. Therefore the conditions of Lemma \ref{l:Lind} hold for $W^{\omega, \sva}_\varphi$ and its kernel $\psi^\sva_\varphi$.

Similarly, we have
\begin{align*}
&{\mathbb V}{\rm ar}(\Xi^{\le l}_{\phi,\Re}) \leq \Eb[(\Xi^{\le l}_{\phi,\Re})^2]\\
=\, & \sum_I \Bigg(\sum_{\genfrac{}{}{0pt}{3}{K_1, K_1 \subset \Omega_\sva \backslash I }{ K_1\cap K_2=\emptyset, |K_2|\in 2\Zb}}
\psi_{\leq l}^\sva(I\cup K_1\cup K_2) \Big(\frac{h^\sva_\cdot}{\lambda^\sva_\cdot}\Big)^{K_1} \Big(\frac{\widetilde{\phi}^\sva_\cdot}{\lambda^\sva_\cdot}\Big)^{K_2}(-1)^{|K_2|/2}\Bigg)^2  \\
\leq\, & \sum_I \Bigg(\sum_{J\subset \Omega_\sva \backslash I} \psi_{\leq l}^\sva(I\cup J) \Big(\frac{|h^\sva_\cdot| + |\widetilde{\phi}^\sva_\cdot|}{\lambda^\sva_\cdot}\Big)^J\Bigg)^2 \\
\leq\, & \sum_I \Big(\sum_{J\subset \Omega_\sva \backslash I} \psi_{\leq l}^\sva(I\cup J)^2  \Big) \Big(\sum_{J\subset \Omega_\sva \backslash I}  \Big(\frac{|h^\sva_\cdot| + |\widetilde{\phi}^\sva_\cdot|}{\lambda^\sva_\cdot}\Big)^{2J}\Big) \\
\leq\, & (1+Ca^2)^{|\Omega_\sva|} \sum_U 2^{|U|}\psi^\sva_{\le l}(U)^2,
\end{align*}
where we used that by \eqref{lambdaxa} and \eqref{phixa} and the assumption $\inf_x \lambda(x)>0$, we can bound $\frac{|h^\sva_\cdot| + |\widetilde{\phi}^\sva_\cdot|}{\lambda^\sva_\cdot}$ uniformly by $Ca$ for some finite $C$. It follows from Lemma \ref{l:spin} that  $\sva^{-k} \psi^\sva(x_1, \ldots, x_k)$ converges in $L^2(\Omega^k)$ for each $k\in\Nb$ and hence ${\mathbb V}{\rm ar}(\Xi^{\le l}_{\phi,\Re})$ is uniformly bounded as $\sva\downarrow 0$. Similarly,
\begin{align*}
{\rm Inf}_x(\psi^\sva_{\phi,\Re}) & = \sum_{I\ni x} \Bigg(\sum_{\genfrac{}{}{0pt}{3}{K_1, K_1 \subset \Omega_\sva \backslash I}{ K_1\cap K_2=\emptyset, |K_2|\in 2\Zb}}
\psi_{\leq l}^\sva(I\cup K_1\cup K_2) \Big(\frac{h^\sva_\cdot}{\lambda^\sva_\cdot}\Big)^{K_1} \Big(\frac{\widetilde{\phi}^\sva_\cdot}{\lambda^\sva_\cdot}\Big)^{K_2}(-1)^{|K_2|/2}\Bigg)^2 \\
& \leq \sum_{U\ni x} 2^{|U|}\psi^\sva_{\le l}(U)^2.
\end{align*}
Since for each $k\in\Nb$, $\sum_{|U|=k} \psi^\sva(U)^2$ converges to a finite integral as shown in \cite[Section 8]{casuzy17}, it is not difficult to see that $\lim_{\sva\downarrow 0}\max_x \sum_{U\ni x} \psi^\sva_{\le l}(U)^2=0$. For more details, see \cite{casuzy17}.

In summary, the conditions of Lemma \ref{l:Lind} also hold for $\Xi^{\le l}_{\phi,\Re}$ and its kernel $\psi^\sva_{\phi,\Re}$. The case
of $\Xi^{\le l}_{\phi,\Im}$ is similar, which completes the proof of Proposition \ref{p:JntCnv}.
\end{proof}

\begin{proof}[Proof of Lemma \ref{l:Lind}]
Label the elements of $\Tb$ by $1, \ldots, |\Tb|$ according to any fixed order.  Write
\[\zeta^j:=(\zeta_i^j)_{1\leq i\leq |\Tb|} \text{ where } \zeta_i^j=\begin{cases} \omega_i& \text{ if } i\leq j, \\ \vartheta_i & \text{ if } i>j, \end{cases} \]
which replaces $\omega_\cdot$ by $\vartheta_\cdot$ one variable at a time. By a telescoping sum, we have
\begin{equation}\label{EbPsi}
\Eb[g(\Psi(\omega))]-\Eb[g(\Psi(\vartheta))] = \sum_{j=1}^{|\Tb|}\left(\Eb[g(\Psi(\zeta^j))]-\Eb[g(\Psi(\zeta^{j-1}))]\right).
\end{equation}
Writing $g((y_i)_{1\leq i\leq n})$ and $\Psi_i((x_j)_{1\leq j\leq |\Tb|})$, we note that $g(\Psi(\zeta^j))$ and $g(\Psi(\zeta^{j-1}))$ differ only in the variable $x_j$. Denote $\varrho^{\omega, \vartheta}_j(x) = g(\Psi(\zeta^{j,x}))$ where $\zeta^{j,x}_j=x$ and $\bar \zeta^{j,x}_i=\zeta^j_i$ for $i\neq j$. By Taylor expansion,
\[
\varrho^{\omega, \vartheta}_j(x)= \varrho^{\omega, \vartheta}_j(0) + x (\varrho^{\omega, \vartheta}_j)'(0)+\frac{x^2}{2}(\varrho^{\omega, \vartheta}_j)''(0)+\Ec_j^{\omega, \vartheta}(x)
\]
where $|\Ec_j^{\omega, \vartheta}(x)|\leq C|x|^3\vert (\varrho^{\omega, \vartheta}_j)'''\vert _\infty$. Note that $ (\varrho^{\omega, \vartheta}_j)'(0)$ and $ (\varrho^{\omega, \vartheta}_j)''(0)$ do not depend on $\omega_j$ or $\vartheta_j$. Therefore using the fact that $\omega_j$ and $\vartheta_j$ have the same mean and variance, we obtain
\begin{align*}
\Eb[g(\Psi(\zeta^j))-g(\Psi(\zeta^{j-1}))] &=\Eb[\varrho^{\omega, \vartheta}_j(\omega_j) -\varrho^{\omega, \vartheta}_j(\vartheta_j)]
=\Eb[\Ec_j^{\omega, \vartheta}(\omega_j)] - \Eb[\Ec_j^{\omega, \vartheta}(\vartheta_j)].
\end{align*}
Substituting into \eqref{EbPsi} then gives
\begin{equation}\label{EgPsibd}
\big|\Eb[g(\Psi(\omega^\sva))]-\Eb[g(\Psi(\vartheta^\sva))] \big|
\leq \sum_{j=1}^{|\Tb|} \Big(\Eb[|\Ec_j^{\omega, \vartheta}(\omega_j)|] + \Eb[|\Ec_j^{\omega, \vartheta}(\vartheta_j)|]\Big).
\end{equation}
Let $\partial_i g$ and $\partial_j \Psi_i$ denote the derivative with respect to the $i$-th and $j$-th component of $g$ and $\Psi_i$, respectively. Note that
\begin{align*}
(\varrho^{\omega, \vartheta}_j)'''(x) & = \sum_{i_1, i_2, i_3=1}^n \partial^3_{i_1, i_2, i_3} g(\Psi(\zeta^{j,x})) \partial_j\Psi_{i_1}(\zeta^{j}) \partial_j\Psi_{i_2}(\zeta^{j}) \partial_j\Psi_{i_2}(\zeta^{j}),
\end{align*}
where we used the fact that $\Psi_i(\zeta^{j,x})$ is linear in $x$.  Bounding the derivatives of $g$ by their sup-norm and applying the triangle inequality, we obtain
\begin{align*}
|\Ec_j^{\omega, \vartheta}(\omega_j)|
\; \leq \; C \; |\omega_j|^3\big\vert (\varrho_j^{\omega, \vartheta})'''\big\vert_\infty
\leq C_{g, n}\sum_{i=1}^n \big|\omega_j \partial_j \Psi_i(\zeta^j)\big|^3,
\end{align*}
Substituting this bound into \eqref{EgPsibd} then gives
\begin{equation}\label{upsibd}
\big|\Eb[g(\Psi(\omega^\sva))]-\Eb[g(\Psi(\vartheta^\sva))] \big|
\leq C_{g,n} \sum_{i=1}^n \sum_{j=1}^{|\Tb|} \Eb\Big[\big|\omega_j \partial_j \Psi_i(\zeta^j)\big|^3\Big].
\end{equation}
Note that
$$
\Eb\big[\big|\omega_j \partial_j \Psi_i(\zeta^j)\big|^2\big] = \sum_{I \ni j} \psi_i(I)^2 = {\rm Inf}_j[\psi_i].
$$
The arguments to bound \eqref{upsibd} are exactly the same as in the proof of \cite[Theorem 2.6]{casuzy17}, starting from (4.17) therein, by applying hypercontractivity to $\Eb\big[\big|\omega_j \partial_j \Psi_i(\zeta^j)\big|^3\big]$ to get a bound of the form
$C_l {\rm Inf}_j[\psi_i]^{3/2}$. The proof there assumed only finite second moments for $\omega$ and $\vartheta$ and hence required a truncation to apply hypercontractivity. Here we have assumed finite third absolute moments, which allows us to apply hypercontractivity directly. Comparing with the statement of \cite[Theorem 2.6]{casuzy17} then gives the desired bound \eqref{Lind}.
\end{proof}

\section{Tightness}\label{s:Tight}

In this section we prove that the laws of $(W^{\omega,\sva},\widetilde{Z}^{\omega,\sva}_{\Omega;\lambda,h},\mu^{\omega,\sva}_{\Omega;\lambda,h})$ are tight in $\Wh\times\Rb\times\Mc_1(\Cc^\alpha_{loc}(\Omega))$.
Tightness of $(W^{\omega,\sva})_{\sva\in(0,1]}$ is standard and can be checked using the tightness criterion of \cite{fumo17} similarly to below. Therefore we omit the proof.
It was shown in Proposition \ref{p:JntCnv} that $\widetilde{Z}^{\omega,\sva}_{\Omega;\lambda,h}$ converges in distribution as $a\downarrow 0$ and hence their laws are tight. 

It remains to show that the family of laws of $(\mu^{\omega,\sva}_{\Omega;\lambda,h})_{\sva\in(0,1]}$ is tight. As random probability measures on the complete separable metric space $\Cc^\alpha_{loc}(\Omega)$, where $\alpha<-1/8$, tightness would follow if we show that
\begin{equation}\label{tightness}
\forall \, \varepsilon>0 \ \ \exists \ {\rm compact} \  K_\varepsilon \subset \Cc^\alpha_{loc}(\Omega) \ \ \mbox{such that }\ \  \sup_{\sva\in (0,1]}\Eb[\mu^{\omega,\sva}_{\Omega;\lambda,h}(K_\varepsilon^c)]\leq \varepsilon.
\end{equation}
This follows from the tightness criteria for random probability measures in \cite[Theorem 4.10]{ka17}. Note that
	\begin{align*}
	\Eb[\mu^{\omega,\sva}_{\Omega;\lambda,h}(K_\varepsilon^c)] \leq \Pb( \widetilde{Z}_{\Omega;\lambda,h}^{\omega,\sva}\leq\eta) +
	\Eb[\mu^{\omega,\sva}_{\Omega;\lambda,h}(K_\varepsilon^c) \ind_{\{\widetilde{Z}_{\Omega;\lambda,h}^{\omega,\sva}\geq\eta\}}].
	\end{align*}
Since $\widetilde{Z}_{\Omega;\lambda,h}^{\omega,\sva}$ converges in distribution to $\Zc^W_{\Omega; \lambda,h}$, which is positive almost surely by Lemma \ref{l:mom}, we can choose $\eta>0$ sufficiently small such that
\[\limsup_{\sva\downarrow 0} \Pb( \widetilde{Z}_{\Omega;\lambda,h}^{\omega,\sva}\leq\eta) \leq \varepsilon. \]
\red{The tightness of $(\mu^{\omega,\sva}_{\Omega;\lambda,h})_{\sva\in(0,1]}$ as $\Mc_1(\Cc^{\alpha}_{loc}(\Omega))$-valued random variables would then follow from the tightness of the average quenched laws $(\Eb[\mu^{\omega,\sva}_{\Omega;\lambda,h}(\cdot);\widetilde{Z}_{\Omega;\lambda,h}^{\omega,\sva}\geq\eta])$ for every $\eta>0$. To verify this, note that for any compact $K\subset \Cc^\alpha_{loc}(\Omega)$,
\begin{align*}
\Eb[\mu^{\omega,\sva}_{\Omega;\lambda,h}(K);\widetilde{Z}_{\Omega;\lambda,h}^{\omega,\sva}\geq\eta] & =  \Eb\left[\frac{\theta_\sva}{\widetilde{Z}_{\Omega;\lambda,h}^{\omega,\sva}}\Et_{\Omega}^{\sva}\left[ e^{\sum_{x\in\Omega_\sva}(\lambda_x^\sva\omega^\sva_x+h_x^\sva)\sigma_x} \ind_{\{ \Phi^\sva_\Omega\in K \}} \right]\ind_{\{\widetilde{Z}_{\Omega;\lambda,h}^{\omega,\sva}\geq\eta\}}\right] \\
&\leq \frac{1}{\eta} \Et_{\Omega}^{\sva}\left[ \theta_\sva\Eb\Big[e^{\sum_{x\in\Omega_\sva}\lambda_x^\sva\omega^\sva_x \sigma_x}\Big] e^{\sum_{x\in\Omega_\sva}h_x^\sva \sigma_x} \ind_{\{ \Phi^\sva_\Omega\in K \}} \right] \\
& \leq \frac{2}{\eta} \Et_{\Omega}^{\sva}\left[e^{\sum_{x\in\Omega_\sva}h_x^\sva \sigma_x} \ind_{\{ \Phi^\sva_\Omega\in K \}} \right] \\
& \leq \frac{2}{\eta}  \Et_{\Omega}^{\sva}\left[e^{\sum_{x\in\Omega_\sva}2h_x^\sva \sigma_x} \right]^{\frac12} \mu^\sva_\Omega(K)^{\frac12} \leq C_\eta \mu^\sva_\Omega(K)^{\frac12},
\end{align*}
where we used that $\theta_\sva\Eb\left[e^{\sum_{x\in \Omega_{\sva}}\lambda_x^\sva \omega_x^\sva\sigma_x}\right]\to 1$ uniformly in $\sigma$ as $\sva\rightarrow 0$ by Lemma \ref{l:Theta},  and
$$
\Et_{\Omega}^{\sva}\left[e^{\sum_{x\in\Omega_\sva}2h_x^\sva \sigma_x} \right] \leq \Et_{\Omega}^{\sva}\left[e^{\sum_{x\in\Omega_\sva}2|h_x^\sva| \sigma_x} \right] \leq \Et_{\Omega}^{\sva}\left[e^{2\Vert h\Vert_\infty \sva^{15/8}\sum_{x\in\Omega_\sva} \sigma_x} \right]
$$
is bounded above as $\sva\downarrow 0$ by the exponential moment bound for the magnetisation random variable \cite[Proposition 3.5]{cagane15}. Since $(\mu^\sva_\Omega)_{\sva\in (0,1]}$ is a tight family of probability measures on $\Cc^\alpha_{loc}(\Omega)$ for any $\alpha<-1/8$, as shown in \cite[Theorem 1.2]{fumo17}, it follows that $(\Eb[\mu^{\omega,\sva}_{\Omega;\lambda,h}(\cdot);\widetilde{Z}_{\Omega;\lambda,h}^{\omega,\sva}\geq\eta])$ is also tight for any $\eta>0$. This concludes the proof of tightness.
}

\section{Singularity}\label{s:Sing}

\subsection{Proof of Theorem \ref{t:Sing}}
In this section we prove Theorem \ref{t:Sing}, which shows that the disordered continuum limit $\mu_{\Omega;\lambda,h}^{W}$ is almost surely singular with respect to the pure continuum limit $\mu_{\Omega}$, but its average with respect to $W$ is absolutely continuous with respect to $\mu_{\Omega}$. We first prove the second statement which is straightforward and similar to \cite[Theorem 1.4]{casuzy16}.

\begin{proof}[Proof of absolute continuity in Theorem \ref{t:Sing}] Note that by \eqref{e:Qchar2} and \eqref{e:Par} and computing characteristic functions, we find that
\[
\Eb\left[\Zc_{\Omega;\lambda,h}^{W}\mu_{\Omega;\lambda,h}^{W}(\cdot)\right]=\mu_{\Omega;0,h}^{W}(\cdot) = \mu_{\Omega; h}(\cdot),
\]
which is the law of the continuum magnetisation field with external field $h$ but without disorder. Since $\Zc_{\Omega;\lambda,h}^{W}>0$ $\Pb$-a.s.\ by Lemma \ref{l:mom}, for any measurable set $A$ with $\mu_{\Omega;0,h}^{W}(A)=0$, we can find a set $E_A$ with
$\Pb(W\in E_A)=1$ such that $\mu_{\Omega;\lambda,h}^{W}(A)=0$ for all $W\in E_A$. Therefore $\Eb\mu_{\Omega;\lambda,h}^{W}(A)=0$ and hence $\Eb\mu_{\Omega;\lambda,h}^{W}$ is absolutely continuous with respect to $\mu_{\Omega; h}$. By \cite[Proposition 1.5]{cagane16} we have that $\mu_{\Omega;h}$ is absolutely continuous with respect to $\mu_{\Omega}$, which completes the proof of the absolute continuity.
We remark that $E_A$ depends on $A$, and there is no contradiction with the claim that almost surely,
$\mu_{\Omega;\lambda,h}^{W}$ is singular with respect to $\mu_\Omega$, because $\cap E_A$ will not be a measurable set with positive probability.
\end{proof}

\begin{proof}[Proof of singularity in Theorem \ref{t:Sing}]
Our proof is based on discrete approximations and techniques for bounding fractional moments that were first developed in the study of the disordered pinning model (see \cite[Chapter 6]{gi11}). We will in fact show that the almost sure singularity of $\mu^W_{\Omega; \lambda, h}$ with respect to $\mu_\Omega$ is a local phenomenon, namely that if $\widehat\Omega$ is any open subset of $\Omega$ and $\chi_{\widehat\Omega}\Phi$ is the restriction of $\Phi\in \Cc^{\alpha}_{loc}(\Omega)$ to $\widehat \Omega$, then almost surely, $\mu^W_{\Omega; \lambda, h}\circ \chi_{\widehat\Omega}^{-1}$ is singular with respect to $\mu_\Omega\circ \chi_{\widehat\Omega}^{-1}$.

Let us assume for the moment that $h\equiv0$. To ease notation, we shall drop the subscripts $\Omega,\lambda,h$ when it is unambiguous. In particular, $\mu^W:=\mu^W_{\Omega; \lambda, h}$ and $\mu:=\mu_\Omega$. Probability and expectation under $\mu$ and $\mu^W$ will be denoted by $\Pt$, $\Et$ and $\Pt^W$, $\Et^W$ respectively; probability and expectation under the law of white noise $W$ will be denoted by $\Pb$ and $\Eb$ respectively.

We first explain the proof strategy. Let $\widehat \Omega$ be any square sub-domain of $\Omega$. Without loss of generality, we may assume $\widehat \Omega= (0,1)^2$. To show that for $\Pb$ almost every $W$, $\mu^W\circ \chi_{\widehat\Omega}^{-1}$ is singular with respect to $\mu\circ \chi_{\widehat\Omega}^{-1}$, if suffices to pick a filtration $(\Fc^\Phi_N)$ on $\Cc_{loc}^\alpha$, which will be generated by functions of the magnetisation field $\Phi$ restricted to $\widehat \Omega$, and show that restricted to $\Fc_N^\Phi$, the Radon-Nikodym derivative of $\mu^W$ with respect to $\mu$,
\begin{equation}\label{QNW}
Q_N^W:=\frac{\d\mu^{W}|_{\Fc_N^\Phi}}{\d\mu|_{\Fc_N^\Phi}},
\end{equation}
converges in $\Pt$-probability to $0$ as $N\to\infty$. To prove this for $\Pb$-a.e.\ $W$, it suffices to show that
$$
\lim_{N\to\infty} \Eb\Et[ (Q_N^W)^{1/2}] = 0.
$$
We then try to approximate $Q_N^W$ by its lattice analogue, $Q_N^{\omega, \sva}$. The difficulty lies in the fact that we only have the convergence of the law of the magnetisation fields, which does not imply the weak convergence of the Radon-Nikodym derivatives $Q_N^{\omega, \sva}$ to $Q_N^W$. To overcome this difficulty and have a quantity more amenable to lattice approximations, we
replace $Q_N^W$ by
\begin{equation}
\Qc_N:=\Eb\left[Q_N^W\Big|\Fc_N^W\right],
\end{equation}
where $\Fc^W_N$ is a $\sigma$-algebra on $\Cc_{loc}^{\alpha'}$ generated by functions of the white noise $W$ restricted to $\widehat\Omega$. However, this is still not suitable for lattice approximation for reasons explained in Remark \ref{r:Radon}. We need to further approximate the random variables that generate $\Fc^W_N$ and $\Fc^\Phi_N$ by random variables that take values in a discrete set $2^{-m}\Zb$ with $m\in\Nb$ chosen arbitrarily large.

We now define the functions of $W$ and $\Phi$ that will generate $\Fc^W_N$ and $\Fc^\Phi_N$, as well as their discrete approximations.
Making the right choices for these functions is crucial for the fractional moment bounds to be carried out later.

Let $N=2^n$. For $i,j=1,...,N$, denote $B_{i,j}^N :=(\frac{i-1}{N},\frac{i}{N})\times(\frac{j-1}{N},\frac{j}{N})$, which partitions $\Omega=(0,1)^2$ into boxes of side length $1/N$. Let $\Phi\in \Cc_{loc}^\alpha$ be arbitrary, and let $W$ be sampled from $\Cc_{loc}^{\alpha'}$ according to the probability measure $\Pb$. We then define
\begin{equation}\label{WPNij}
\begin{gathered}
\Phi_{\lambda, i,j}^N := \Big\langle\Phi \lambda^2,\ind_{B_{i,j}^N}\Big\rangle, \qquad
W_{\lambda, i,j}^N := \Big\langle W,\lambda\ind_{B_{i,j}^N}\Big\rangle, \\
\Fc_N^\Phi :=\sigma\Big(\Big\{\Phi_{\lambda, i,j}^N\Big\}_{i,j=1}^N\Big), \qquad
\Fc_N^W :=\sigma\Big(\Big\{W_{\lambda, i,j}^N\Big\}_{i,j=1}^N\Big),
\end{gathered}
\end{equation}
where $\langle\Phi \lambda^2,\ind_{B_{i,j}^N}\rangle$ is well-defined and continuous in $\Phi \in \Cc^\alpha_{loc}$ for $\alpha \in (-1,0)$ by Lemma \ref{l:test} and the fact that $\Phi\to \lambda^2 \Phi$ is a continuous map from $\Cc^\alpha_{loc}$ to itself because $\lambda^2 \in C^1(\Omega)$ (see e.g.~\cite[Proposition 2.19]{fumo17}). This induces a $\sigma$-algebra $\Fc_N^\Phi$ on $\Cc^\alpha_{loc}$. As a stochastic integral, the mapping $\langle W, \lambda \ind_{B_{i,j}^N}\rangle$ is well-defined for $\Pb$-a.e.~$W\in \Cc^{\alpha'}$, and hence it induces a $\sigma$-algebra $\Fc_N^W$ on $\Cc^{\alpha'}_{loc}$ if we complete the Borel $\sigma$-algebra on $\Cc^{\alpha'}$ with sets of $\Pb$ measure $0$. Let $\Fc_N := \Fc_N^\Phi \times \Fc_N^W$ denote the product $\sigma$-algebra on $\Cc_{loc}^\alpha \times \Cc_{loc}^{\alpha'}$. Note that if $(W, \Phi)$ are sampled according to $\Pb\times\Pt$, while given $W$, $\Phi^W$ is sampled according to $\Pt^W$, then $\Qc_N$ is simply the Radon-Nikodym derivative of the law of $(W^N_{\lambda, i,j}, \Phi^{W,N}_{\lambda,i,j})_{i,j=1}^N$ with respect to the law of $(W^N_{\lambda, i,j}, \Phi^N_{\lambda, i,j})_{i,j=1}^N$.

Next, we further approximate the random variables $W^N_{\lambda, i,j}, \Phi^{N}_{\lambda, i,j}$, and $\Phi^{W,N}_{\lambda, i,j}$ by discrete valued functions
as follows. For $m\in\Nb$, define
\begin{equation}\label{WPNmij}
\begin{gathered}
\Phi_{\lambda, i,j}^{N, m} := 2^{-m} \big\lfloor 2^m \Phi_{\lambda, i,j}^{N} \big\rfloor,  \quad \quad
W_{\lambda, i,j}^{N,m} := 2^{-m} \big\lfloor 2^m W_{\lambda, i,j}^N \big\rfloor,  \\
\Fc_{N,m}^\Phi :=\sigma\Big(\Big\{\Phi_{\lambda, i,j}^{N,m}\Big\}_{i,j=1}^N\Big), \qquad
\Fc_{N,m}^W :=\sigma\Big(\Big\{W_{\lambda, i,j}^{N,m}\Big\}_{i,j=1}^N\Big),
\end{gathered}
\end{equation}
and let $\Fc_{N,m} := \Fc_{N,m}^\Phi \times \Fc_{N,m}^W$. Note that $(\Fc^\Phi_{N,m})_{m\in\Nb}$ and $(\Fc^W_{N,m})_{m\in\Nb}$ form filtrations that generate $\Fc_N^\Phi$ and $\Fc_N^W$ respectively. Denote
\begin{equation}\label{QN}
\Qc_{N,m} = \Eb\Et[\Qc_N| \Fc_{N,m}]],
\end{equation}
which is just the Radon-Nikodym derivative of the law of $(W^{N,m}_{\lambda,i,j}, \Phi^{W,N,m}_{\lambda, i,j})_{i,j=1}^N$ with respect to the law of $(W^{N,m}_{\lambda, i,j}, \Phi^{N,m}_{\lambda, i,j})_{i,j=1}^N$, where $(W, \Phi, \Phi^W)$ are sampled according to $\Pb\times\Pt \times\Pt^W$. To prove that for $\Pb$ almost every $W$, $\mu^W$ is singular with respect to $\mu$, it then suffices to show that
\begin{equation}\label{QNlim}
\lim_{N\to\infty} \lim_{m\to\infty} \Eb\Et\Big[\Qc_{N,m}^{1/2}\Big] = 0.
\end{equation}

We now define the lattice analogue of $\Qc_{N,m}$. We are free to choose the law of the disorder $\omega$ as long as Theorem \ref{t:JntCnv} can be applied. Therefore we let $\omega^\sva_x:= \sva^{-1}\int_{S_\sva(x)}W(\d y)$ so that $(\omega^\sva_x)_{x\in\Omega_\sva}$ are standard Gaussian random variables defined from the white noise $W$.  Instead of working with the rescaled piecewise constant magnetisation field $\Phi^\sva$ defined in \eqref{Phiomega}, we need to work with
\begin{equation}\label{Wom}
W^{\sva} := \sum_{x\in\Omega_\sva}\sva \omega_x^\sva\delta_x \qquad \mbox{and} \qquad \widetilde{\Phi}^a := \sum_{x\in\Omega_\sva}\sva^{15/8}\sigma_x\delta_x,
\end{equation}
where the spin configuration $\sigma$ is sampled according to $\Pt^a_\Omega$. If $\sigma$ is sampled according to $\Pt^{\omega, \sva}_{\Omega; \lambda ,0}$, we will denote the resulting magnetisation field by $\widetilde{\Phi}^{\omega, \sva}$ instead. We can apply the same operations as in \eqref{WPNij} and \eqref{WPNmij} to obtain $W^{\sva; N}_{\lambda, i,j}$, $\widetilde \Phi^{\sva;N}_{\lambda, i,j}$, $\widetilde \Phi^{\omega, \sva;N}_{\lambda, i,j}$, $W_{\lambda, i,j}^{a; N, m} $, $\widetilde{\Phi}_{\lambda, i,j}^{a; N, m}$ and $\widetilde{\Phi}_{\lambda, i,j}^{\omega, \sva; N, m}$.

For $N=2^n$, $m\in\Nb$ and $\sva\in(0,1]$, let $\Fc_N^\sva$ and $\Fc_{N,m}^\sva$ be the $\sigma$-fields generated by $(W^{\sva; N}_{\lambda, i,j}, \widetilde \Phi^{\sva;N}_{\lambda, i,j})_{i,j=1}^N$ and $(W^{\sva; N, m}_{\lambda, i,j}, \widetilde{\Phi}^{\sva; N, m}_{\lambda, i,j})_{i,j=1}^N$, respectively then, define
\begin{equation}\label{QNa}
\Qc_{N}^\sva:=\Eb\Et\left[\frac{\d\mu^{\omega,\sva}}{\d\mu^{\sva}}\Big|\Fc_{N}^\sva\right] \quad \mbox{and} \quad \Qc_{N,m}^\sva:=\Eb\Et\left[\frac{\d\mu^{\omega,\sva}}{\d\mu^{\sva}}\Big|\Fc_{N,m}^\sva\right].
\end{equation}
Note that $\Qc_N^\sva$ is the Radon-Nikodym derivative of the law of $(W^{\sva; N}_{\lambda, i,j}, \widetilde{\Phi}^{\omega, \sva; N}_{\lambda, i,j})_{i,j=1}^N$ with respect to the law of $(W^{\sva; N}_{\lambda, i,j}, \widetilde{\Phi}^{\sva; N}_{\lambda, i,j})_{i,j=1}^N$, while $\Qc_{N,m}^\sva$ is the Radon-Nikodym derivative of the law of $(W^{\sva; N, m}_{\lambda, i,j}, \widetilde{\Phi}^{\omega, \sva; N, m}_{\lambda, i,j})_{i,j=1}^N$ with respect to the law of $(W^{\sva; N, m}_{\lambda, i,j}, \widetilde{\Phi}^{\sva; N, m}_{\lambda, i,j})_{i,j=1}^N$.
\medskip

To carry through the lattice approximation, we need the following convergence result.
\begin{lem}\label{l:aprox}
Denote the laws of $(W^{N,m}_{\lambda,i,j}, \Phi^{N,m}_{\lambda, i,j})_{i,j=1}^N$ and $(W^{N,m}_{\lambda,i,j}, \Phi^{W,N,m}_{\lambda, i,j})_{i,j=1}^N$ by  $\mu_{N,m}$ and $\nu_{N,m}$ respectively, where $(W, \Phi, \Phi^W)$ is sampled from $\Pb\times\Pt \times\Pt^W$. Similarly, let $\mu^\sva_{N,m}$ and $\nu^\sva_{N,m}$ denote the law of $(W^{a; N,m}_{\lambda,i,j}, \widetilde{\Phi}^{\sva; N,m}_{\lambda, i,j})_{i,j=1}^N$ and $(W^{a; N,m}_{\lambda,i,j}, \widetilde{\Phi}^{\omega, \sva; N,m}_{\lambda, i,j})_{i,j=1}^N$, where $(W^\sva, \widetilde{\Phi}^\sva, \widetilde{\Phi}^{\omega, \sva})$ is sampled from $\Pb\times\Pt^\sva_\Omega \times\Pt^{\omega, \sva}_{\Omega; \lambda, 0}$. For any $N$ and $m$, we have $\mu^\sva_{N,m} \Rightarrow \mu_{N,m}$ and $\nu^\sva_{N,m} \Rightarrow \nu_{N,m}$ as $\sva\downarrow 0$.
\end{lem}
Since $\mu_{N,m}$, $\nu_{N,m}$, $\mu_{N,m}^\sva$ and $\nu^\sva_{N,m}$ are all supported on the discrete set $(2^{-m}\Zb)^{N^2}$, we have
\begin{align}
\Eb\Et\left[\Qc_{N,m}^{1/2}\right]  & = \sum_{x\in (2^{-m}\Zb)^{N^2}} \Big(\frac{\nu_{N,m}(x)}{\mu_{N,m}(x)}\Big)^{1/2} \mu_{N,m}(x) \notag\\
& \leq \liminf_{\sva\downarrow 0} \sum_{x\in (2^{-m}\Zb)^{N^2}} \Big(\frac{\nu^\sva_{N,m}(x)}{\mu^\sva_{N,m}(x)}\Big)^{1/2} \mu^\sva_{N,m}(x) = \liminf_{\sva\downarrow 0}\Eb\Et\left[(\Qc_{N,m}^\sva)^{1/2}\right].\label{e:Fat}
\end{align}
Therefore to prove \eqref{QNlim}, it suffices to show that
\begin{lem} \label{l:Null}
We have
\begin{equation}\label{e:Null}
\lim_{N\rightarrow \infty}\lim_{m\to\infty} \liminf_{\sva\downarrow 0}\Eb\Et\left[(\Qc_{N,m}^\sva)^{1/2}\right]=0.
\end{equation}
\end{lem}
\noindent
This would conclude the proof of the almost sure singularity of $\mu^W$ with respect to $\mu$ when $h\equiv0$. We will prove Lemmas \ref{l:aprox} and \ref{l:Null} in the next two subsections.

To extend the singularity result to general $h\in C^1(\Omega)$, notice that the Radon-Nikodym derivative of the joint law of $(W, \Phi)$ under $\Pb\Pt^{\omega, \sva}_{\Omega; \lambda, h+\Delta}$ with respect to the joint law of $(W, \Phi)$ under $\Pb\Pt^{\omega, \sva}_{\Omega; \lambda, h}$ is given by
\begin{align*}
\frac{\d \mu^{\omega,\sva}_{\Omega;\lambda,h+\Delta}}{\d \mu^{\omega,\sva}_{\Omega;\lambda,h}}(\Phi)
\; = \; \exp(\langle \Phi, \Delta\rangle)\frac{\widetilde{Z}^{\omega,\sva}_{\Omega;\lambda,h}}{\widetilde{Z}^{\omega,\sva}_{\Omega;\lambda,h+\Delta}} =: Y,
\end{align*}
which is uniformly integrable with respect to $\Pb\Pt^{\omega, \sva}_{\Omega; \lambda, h}$ because for any $C>0$,
\begin{align*}
\Eb\Et^{\omega, \sva}_{\Omega; \lambda, h}\big[Y \ind_{Y>C}\big]
& = \Eb\Et^{\omega, \sva}_{\Omega; \lambda, h+\Delta}\big[\ind_{Y>C}\big] \\
& \leq C^{-1} \Eb\Et^{\omega, \sva}_{\Omega; \lambda, h+\Delta}\Bigg[e^{\langle \Phi, \Delta\rangle} \frac{\widetilde{Z}^{\omega,\sva}_{\Omega;\lambda,h}}{\widetilde{Z}^{\omega,\sva}_{\Omega;\lambda,h+\Delta}} \Bigg] \\
& = C^{-1} \Eb\Et^{\sva}_{\Omega}\Bigg[e^{\sum_x \sigma_x \Delta^\sva_x} \cdot \frac{e^{\sum_x \sigma_x (\lambda^\sva_x \omega^\sva_x+ h^\sva_x+\Delta^\sva_x)}}{Z^{\omega, \sva}_{\Omega;\lambda, h+\Delta}}\cdot \frac{\widetilde{Z}^{\omega,\sva}_{\Omega;\lambda,h}}{\widetilde{Z}^{\omega,\sva}_{\Omega;\lambda,h+\Delta}} \Bigg] \\
& = C^{-1} \Eb\Bigg[\frac{\widetilde{Z}^{\omega,\sva}_{\Omega;\lambda,h+2\Delta}\,\widetilde{Z}^{\omega,\sva}_{\Omega;\lambda,h}}{\big(\widetilde{Z}^{\omega,\sva}_{\Omega;\lambda,h+\Delta}\big)^2}\Bigg] \\
& \leq C^{-1} \Eb\Big[\big(\widetilde{Z}^{\omega,\sva}_{\Omega;\lambda,h+2\Delta}\big)^3\Big]^{\frac13} \Eb\Big[\big(\widetilde{Z}^{\omega,\sva}_{\Omega;\lambda,h}\big)^3\Big]^{\frac13} \Eb\Big[\big(\widetilde{Z}^{\omega,\sva}_{\Omega;\lambda, h+\Delta}\big)^{-6}\Big]^{\frac13},
\end{align*}
where the moments are uniformly bounded by Lemma \ref{l:mom}. Since $h$ and $\Delta$ can be chosen arbitrarily, we can now apply Lemma \ref{l:Radon} to conclude that as $\sva\downarrow 0$, the limiting law of $(W, \Phi)$ under $\Pb \Pt^W_{\Omega; \lambda, h_1}$ is mutually absolutely continuous with respect to the law of $(W, \Phi)$ under $\Pb \Pt^W_{\Omega; \lambda, h_2}$. In particular, for $\Pb$-a.e.\ $W$, $\Pt^W_{\Omega; \lambda, h_1}$ is mutually absolutely continuous with respect to $\Pt^W_{\Omega; \lambda, h_2}$. Since $\Pt^W_{\Omega; \lambda, 0}$ is singular with respect to $\Pt_\Omega$, so must be $\Pt^W_{\Omega; \lambda, h}$  for any $h\in C^1(\Omega)$.
\end{proof}

\subsection{Proof of Lemma \ref{l:aprox}}\label{s:unifint}

\begin{proof}[Proof of Lemma \ref{l:aprox}]
By Theorem \ref{t:JntCnv}, $(W^{\sva}, \Phi^\sva, \Phi^{\omega, \sva})$ converges in distribution \red{to} $(W, \Phi, \Phi^W)$ as $\sva\downarrow 0$, where $\Phi^\sva$ and $\Phi^{\omega, \sva}$ are the \red{rescaled piecewise constant magnetisation fields.} Define $(\Phi^{\sva; N}_{\lambda, i,j}, \Phi^{\omega, \sva; N}_{\lambda, i,j})_{i, j=1}^N$ from $(\Phi^\sva, \Phi^{\omega, \sva})$ as in \eqref{WPNij} and \eqref{WPNmij}. By the continuous mapping theorem, $(W^{\sva; N}_{\lambda, i,j}, \Phi^{\sva; N}_{\lambda, i,j}, \Phi^{\omega, \sva; N}_{\lambda, i,j})_{i, j=1}^N$ converges in distribution to $(W^{N}_{\lambda, i,j}, \Phi^{N}_{\lambda, i,j}, \Phi^{W,N}_{\lambda, i,j})_{i,j=1}^N$.
Since $\Phi^\sva$ and $\Phi^{\omega, \sva}$ are smeared out versions of $\widetilde\Phi^\sva$ and $\widetilde\Phi^{\omega, \sva}$ with the latter supported on $\Omega_\sva$, we have that
\begin{align*}
\left|\widetilde{\Phi}^{\sva,N}_{\lambda,i,j}-\Phi^{\sva,N}_{\lambda,i,j}\right|
=\left|\sum_{x\in B^{N, \sva}_{i,j}}\!\sva^{15/8}\sigma_x\Big(\lambda^2(x)-\sva^{-2}\int_{S_\sva(x)} \lambda^2(y) {\rm d}y\Big)\right|
& \leq 2\sva^{7/8} N^{-2}\Vert \lambda \Vert_\infty \Vert \lambda'\Vert_\infty
\end{align*}
converges to $0$ deterministically. Therefore $(W^{\sva; N}_{\lambda, i,j}, \widetilde{\Phi}^{\sva; N}_{\lambda, i,j}, \widetilde{\Phi}^{\omega, \sva; N}_{\lambda, i,j})_{i, j=1}^N$ also converges in distribution to $(W^{N}_{\lambda, i,j}, \Phi^{N}_{\lambda, i,j}, \Phi^{W,N}_{\lambda, i,j})_{i,j=1}^N$.
\medskip

To show that the law of the discretised random variables also converge, it only remains to show that $(W^{N}_{\lambda, i,j}, \Phi^{N}_{\lambda, i,j}, \Phi^{W,N}_{\lambda, i,j})_{i,j=1}^N$ are continuous random variables. For $\Phi^{N}_{\lambda, i,j}$, this follows similarly to \cite[Theorem 1.3]{cagane15}, which is based on bounds on its Fourier transform; for $W^{N}_{\lambda, i,j}$, this follows because it is Gaussian. We show in Lemma \ref{l:UniInt} below that $\Qc_N^\sva$, the Radon-Nikodym derivative of the law of $(W^{\sva; N}_{\lambda, i,j}, \widetilde{\Phi}^{\omega, \sva; N}_{\lambda, i,j})$ with respect to $(W^{\sva; N}_{\lambda, i,j}, \widetilde{\Phi}^{\sva; N}_{\lambda, i,j})$, is uniformly integrable as $\sva\downarrow 0$, and hence by Lemma \ref{l:Radon}, the law of the limit $(W^{N}_{\lambda, i,j}, \Phi^{W, N}_{\lambda, i,j})$ is absolutely continuous with respect to the law of $(W^{N}_{\lambda, i,j}, \Phi^{N}_{\lambda, i,j})$. Therefore $(\Phi^{W,N}_{\lambda,i,j})_{i,j=1}^N$ are also continuous random variables, which concludes the proof.
\end{proof}

We now prove the uniform integrability of $(\Qc_N^\sva)_{\sva\in(0,1)}$ defined in \eqref{QNa}.
\begin{lem}\label{l:UniInt}
For each $N\in\Nb$, $(\Qc_N^\sva)_{\sva\in(0,1)}$ is uniformly integrable with respect to $\Eb\Pt$.
\end{lem}
\begin{proof}
Abbreviate $\mu^{\omega, \sva}_{\Omega; \lambda, 0}$ and $\mu^{\omega, \sva}_{\Omega; 0, 0}$ by $\mu^{\omega, \sva}$ and $\mu^\sva$ respectively. We have that
\[\frac{\d\mu^{\omega,\sva}}{\d\mu^{\sva}}=(1+o(1))\frac{H^{\omega,\sva}}{\widetilde{Z}^{\omega,\sva}} \quad \text{where} \quad H^{\omega,\sva}:=\exp\left(\sum_x\sigma_x\omega_x^\sva\lambda_x^\sva-\frac{1}{2}(\lambda_x^\sva)^2\right)\]
and $1+o(1)=\theta_\sva e^{\frac{1}{2}\sum_{x\in\Omega_\sva} (\lambda_x^\sva)^2}= e^{-\frac{1}{2} \sva^{-\frac{1}{4}} \Vert \lambda\Vert_{L^2}^2+\frac{1}{2}\sum_{x\in\Omega_\sva} (\lambda_x^\sva)^2}$ is deterministic.

Note that given $\varepsilon>0$, we can write
\begin{equation}\label{QcNde}
\Qc_N^\sva = \Eb\Et\left[\frac{H^{\omega,\sva}}{\widetilde{Z}^{\omega,\sva}}\ind_{\widetilde{Z}^{\omega,\sva}\leq \varepsilon}\Big|\Fc_N^\sva\right] + \Eb\Et\left[\frac{H^{\omega,\sva}}{\widetilde{Z}^{\omega,\sva}}\ind_{\widetilde{Z}^{\omega,\sva}> \varepsilon}\Big|\Fc_N^\sva\right],
\end{equation}
where the first term has expectation
$$
\Eb\Et\left[\Eb\Et\left[\frac{H^{\omega,\sva}}{\widetilde{Z}^{\omega,\sva}}\ind_{\widetilde{Z}^{\omega,\sva}\leq \varepsilon}\Big|\Fc_N^\sva\right]\right]
 = \Eb\Et\left[\frac{H^{\omega,\sva}}{\widetilde{Z}^{\omega,\sva}}\ind_{\widetilde{Z}^{\omega,\sva}\leq \varepsilon}\right]
 = (1+o(1))\Pb(\widetilde{Z}^{\omega,\sva}\leq \varepsilon)
$$
since $\widetilde{Z}^{\omega,\sva}=(1+o(1))\Et[H^{\omega,\sva}]$ by Lemma \ref{l:Theta}. This expectation can be made arbitrarily small uniformly in $\sva$ by choosing $\varepsilon>0$ small, because $\widetilde{Z}^{\omega,\sva}$ converges in distribution to $\Zc^W_{\Omega; \lambda, h}$ which is a.s.\ positive by Lemma \ref{l:mom}. The second term in \eqref{QcNde} is bounded by
$\varepsilon^{-1} \Eb\Et[H^{\omega,\sva}|\Fc_N^\sva]$. Combining these two observations, it follows that to prove uniform integrability of $(\Qc_N^\sva)_{\sva\in (0,1)}$, it suffices to prove that
\begin{equation}\label{HUI}
(\Eb\Et[H^{\omega,\sva}|\Fc_N^\sva])_{\sva\in(0,1)} \quad \mbox{is uniformly integrable}.
\end{equation}

We will use the following lemma whose proof is standard and will be omitted.
\begin{lem}\label{l:ConGau}
Suppose $(X_k)_{k=1}^L$, are independent, centred Gaussian random variables with variances $(\varsigma_k)_{k=1}^L$. For $M\in\Rb$ fixed, $(X_k)_{k=1}^{L-1}$ conditioned on the event $\sum_{k=1}^L X_k=M$ has a multivariate Gaussian distribution given by
\begin{equation}\label{Gauss1}
\Eb\left[X_j\Bigg|\sum_{k=1}^L X_k=M\right]=\frac{M\varsigma_j}{\sum_{k=1}^L\varsigma_k}
\end{equation}
and
\begin{equation}\label{Gauss2}
\Eb\left[X_j X_l\Bigg|\sum_{k=1}^L X_k=M\right]=\begin{cases}\frac{\varsigma_j\sum_{k\neq j}\varsigma_k}{\sum_{k=1}^L\varsigma_k} & \text{if } j=l,\\
\frac{\varsigma_j\varsigma_l\left(M^2-\sum_{k=1}^L\varsigma_k\right)}{\left(\sum_{k=1}^L\varsigma_k\right)^2} & \text{if } j\neq l.\end{cases}
\end{equation}
\end{lem}

Write $B_{i,j}^{N,\sva}$ for the discretisation of the box $B_{i,j}^N$ in $\sva \Zb^2$. Conditioning on the spin configuration $\sigma$ and
using independence of $\omega^\sva_\cdot$ in disjoint boxes, we have that $\Eb\Et\left[H^{\omega,\sva}\big|\Fc_N\right]$ equals
\begin{align*}
&\Et\Bigg[\Eb\Bigg[\exp\left(\sum_{x\in\Omega_\sva}\sigma_x\omega_x^\sva\lambda_x^\sva-\frac{1}{2}(\lambda_x^\sva)^2\right)\Bigg|\Bigg\{a\!\!\!\sum_{x\in B_{i,j}^{N,\sva}}\lambda(x)\omega_x^\sva=W_{\lambda, i,j}^{\sva,N}\Bigg\}_{i,j=1}^N\Bigg]\Bigg|\left\{\widetilde{\Phi}_{\lambda, i,j}^{\sva,N}\right\}_{i,j=1}^N\Bigg] \\
=\, &  \Et\Bigg[\prod_{i,j=1}^N\Eb\Bigg[\exp\Bigg(\sum_{x\in B_{i,j}^{N,\sva}}\sigma_x\omega_x^\sva\lambda_x^\sva-\frac{1}{2}(\lambda_x^\sva)^2\Bigg)\Bigg|\, a\!\!\!\sum_{x\in B_{i,j}^{N,\sva}}\lambda(x)\omega_x^\sva=W_{\lambda, i,j}^{\sva,N}\Bigg]\Bigg|\Bigg\{\widetilde{\Phi}_{\lambda, i,j}^{\sva,N}\Bigg\}_{i,j=1}^N\Bigg].
\end{align*}
Let $\Lambda_{i,j}^N:=\sum_{x\in B_{i,j}^{N,\sva}}\lambda(x)^2$ and
\[
\Pb_{i,j}^N(\cdot):=\Pb\Bigg(\cdot\Bigg|\, a\!\!\!\sum_{x\in B_{i,j}^{N,\sva}}\lambda(x)\omega_x^\sva=W_{\lambda, i,j}^{\sva,N}\Bigg).
\]
Fix any vertex $z\in B_{i,j}^{N,\sva}$. By Lemma \ref{l:ConGau}, the laws of $(\lambda(x)\omega_x)_{x\in B_{i,j}^{N,\sva}\setminus\{z\}}$ with respect to $\Pb_{i,j}^N$ are Gaussian with means and covariances given by
\begin{align*}
\Eb_{i,j}^N\left[\lambda(x)\omega_x^\sva\right] & =\frac{\lambda^2(x)\sva^{-1}W_{\lambda, i,j}^{\sva,N}}{\Lambda_{i,j}^N},\\
\text{Cov}_{\Pb_{i,j}^N}\left(\lambda(x)\omega_x^\sva,\lambda(y)\omega_y^\sva\right) & = \begin{cases}\frac{\lambda(x)^2(\Lambda_{i,j}^N-\lambda(x)^2)}{\Lambda_{i,j}^N} & \text{if } x=y, \\ -\frac{\lambda(x)^2\lambda(y)^2}{\Lambda_{i,j}^N} & \text{if } x\neq y. \end{cases}
\end{align*}
Recall that
\begin{equation*}
\widetilde{\Phi}_{\lambda,i,j}^{\sva,N}=\langle \widetilde{\Phi}^\sva, \lambda^2\ind_{B_{i,j}^N}\rangle=\sum_{x\in B_{i,j}^N}\sva^{15/8}\lambda^2(x)\sigma_x.
\end{equation*}
Since under $\Pb_{i,j}^N$, $\sum_{x\in B_{i,j}^{N,\sva}}\sigma_x\omega_x^\sva\lambda_x^\sva$ (recall that $\lambda^\sva_x=\sva^{7/8}\lambda(x)$) is a Gaussian random variable, it has exponential moment
\begin{align*}
& \Eb_{i,j}^N\left[\exp\left(\sum_{x\in B_{i,j}^{N,\sva}}\sigma_x\omega_x^\sva\lambda_x^\sva\right)\right] \\
=\, &  \exp\Bigg(\sum_{x\in B_{i,j}^{N,\sva}}\!\!\! \frac{\sigma_x\sva^{7/8}\lambda(x)^2\sva^{-1}W_{\lambda, i,j}^{\sva,N}}{\Lambda_{i,j}^N}+\frac{1}{2}\!\!\!\sum_{x\in B_{i,j}^{N,\sva}} \!\!\! \sva^{7/4}\lambda(x)^2
-\frac{1}{2}\!\!\! \! \sum_{x,y\in B_{i,j}^{N,\sva}}\!\!\! \frac{\sigma_x\sigma_y\lambda(x)^2\lambda(y)^2\sva^{7/4}}{\Lambda_{i,j}^N}\Bigg) \\
=\, &  e^{\frac{1}{2}\sum_{x\in B_{i,j}^{N,\sva}}(\lambda_x^\sva)^2}\exp\left(\frac{\widetilde{\Phi}_{\lambda, i,j}^{\sva,N}W_{\lambda, i,j}^{\sva,N}}{\sva^2\Lambda_{i,j}^N}-\frac{1}{2}\frac{(\widetilde{\Phi}_{\lambda, i,j}^{\sva,N})^2}{\sva^2\Lambda_{i,j}^N}\right).
\end{align*}
We note here that $\Lambda_{i,j}^{\sva,N}:=\sva^2\Lambda_{i,j}^N$ converges to $\Vert \lambda\Vert_{L^2(B_{i,j}^N)}^2$ as $\sva\downarrow 0$ and the exponential prefactor compensates with the tilting $-(\lambda_x^\sva)^2/2$ appearing in the definition of $H^{\omega,\sva}$. In particular,
\begin{align}\label{e:CRND}
\Eb\Et\left[H^{\omega,\sva}\big|\Fc_N^\sva\right] & = \prod_{i,j=1}^N\exp\left(\frac{\widetilde{\Phi}_{\lambda, i,j}^{\sva,N}W_{\lambda, i,j}^{\sva,N}}{\Lambda_{i,j}^{\sva,N}}-\frac{1}{2}\frac{(\widetilde{\Phi}_{\lambda, i,j}^{\sva,N})^2}{\Lambda_{i,j}^{\sva,N}}\right),
\end{align}
where conditioned on $\sigma$, $\frac{\widetilde{\Phi}_{\lambda, i,j}^{\sva,N}W_{\lambda, i,j}^{\sva,N}}{\Lambda_{i,j}^{\sva,N}}$ is a centered Gaussian random variable with variance $\frac{(\widetilde{\Phi}_{\lambda, i,j}^{\sva,N})^2}{\Lambda_{i,j}^{\sva,N}}$.

For $K>0$ let
\[
\Ac_K:=\{\Eb\Et\left[H^{\omega,\sva}\big|\Fc_N^\sva\right]>K\} \quad \text{and} \quad
\Bc_K:= \left\{\prod_{i,j=1}^N\exp\left(\frac{(\widetilde{\Phi}_{\lambda, i,j}^{\sva,N})^2}{\Lambda_{i,j}^{\sva,N}}\right)>K^{1/2}\right\}.
\]
We want to show that $\lim_{K\rightarrow \infty}\sup_{\sva\in(0,1]}\Eb\Et\left[\Eb\Et\left[H^{\omega,\sva}\big|\Fc_N^\sva\right]\ind_{\Ac_K}\right]=0$. We have that
\begin{equation}\label{e:UnIn}
\begin{aligned}
& \Eb\Et\left[\Eb\Et\left[H^{\omega,\sva}\big|\Fc_N^\sva\right]\ind_{\Ac_K}\right] \\
\leq\ &  \Eb\Et\left[\Eb\Et\left[H^{\omega,\sva}\big|\Fc_N^\sva\right]\ind_{\Bc_K}\right] +\Eb\Et\left[\Eb\Et\left[H^{\omega,\sva}\big|\Fc_N^\sva\right]\ind_{\Ac_K}\ind_{\Bc_K^c}\right].
\end{aligned}
\end{equation}
Using \eqref{e:CRND} and that $W_{\lambda, i,j}^{\sva,N}$ are independent, centred Gaussians with variances $\Lambda_{i,j}^{\sva,N}$, we can apply Markov's inequality to bound the final term in \eqref{e:UnIn} by
\begin{align*}
& \frac{1}{K}\Eb\Et\left[\prod_{i,j=1}^N\exp\left(\frac{2\widetilde{\Phi}_{\lambda, i,j}^{\sva,N}W_{\lambda, i,j}^{\sva,N}}{\Lambda_{i,j}^{\sva,N}}-\frac{(\widetilde{\Phi}_{\lambda, i,j}^{\sva,N})^2}{\Lambda_{i,j}^{\sva,N}}\right)\ind_{\Bc_K^c}\right] \\
=\, &  \frac{1}{K}\Et\left[\prod_{i,j=1}^N\Eb\left[\exp\left(\frac{2\widetilde{\Phi}_{\lambda, i,j}^{\sva,N}W_{\lambda, i,j}^{\sva,N}}{\Lambda_{i,j}^{\sva,N}}\right)\right]\exp\left(-\frac{(\widetilde{\Phi}_{\lambda, i,j}^{\sva,N})^2}{\Lambda_{i,j}^{\sva,N}}\right)\ind_{\Bc_K^c}\right] \\
=\, &  \frac{1}{K}\Et\left[\prod_{i,j=1}^N\exp\left(\frac{(\widetilde{\Phi}_{\lambda, i,j}^{\sva,N})^2}{\Lambda_{i,j}^{\sva,N}}\right)\ind_{\Bc_K^c}\right]
\leq K^{-1/2}
\end{align*}
uniformly in $\sva$. Similarly, the first term in \eqref{e:UnIn} is equal to
$$
\begin{aligned}
&\Eb\Et\left[\ind_{\Bc_K}\prod_{i,j=1}^N\exp\left(\frac{\widetilde{\Phi}_{\lambda, i,j}^{\sva,N}W_{\lambda, i,j}^{\sva,N}}{\Lambda_{i,j}^{\sva,N}}-\frac{1}{2}\frac{(\widetilde{\Phi}_{\lambda, i,j}^{\sva,N})^2}{\Lambda_{i,j}^{\sva,N}}\right)\right] \\
=\, & \Et\left[\ind_{\Bc_K}\prod_{i,j=1}^N\Eb\left[\exp\left(\frac{\widetilde{\Phi}_{\lambda, i,j}^{\sva,N}W_{\lambda, i,j}^{\sva,N}}{\Lambda_{i,j}^{\sva,N}}\right)\right]\exp\left(-\frac{1}{2}\frac{(\widetilde{\Phi}_{\lambda, i,j}^{\sva,N})^2}{\Lambda_{i,j}^{\sva,N}}\right)\right]  \\
=\, & \Pt(\Bc_K) \leq  2\sum_{i,j=1}^N\frac{\Et\left[(\widetilde{\Phi}_{\lambda, i,j}^{\sva,N})^2\right]}{\Lambda_{i,j}^{\sva,N}\log(K)}
\end{aligned}
$$
which converges to $0$ uniformly in $\sva$ as $K\rightarrow \infty$ since $\langle \widetilde{\Phi}^\sva,\lambda^2\ind_{B^N_{i,j}}\rangle$ has finite exponential moments uniformly in $\sva$ as shown in \cite[Proposition 3.5]{cagane15}.
\end{proof}

\subsection{Proof of Lemma \ref{l:Null}}\label{s:fracmom}
Our proof of Lemma \ref{l:Null} follows the fractional moment method first developed in the study of the disordered pinning model (see \cite[Chapter 6]{gi11}).  It will become clear from the proof why we chose to work with $W^{\sva,N}_{\lambda, i,j}$ and $\widetilde{\Phi}^{\sva,N}_{\lambda, i,j}$ as defined in \eqref{WPNij}.

Recall that $\widetilde{\Phi}^\sva=\sum_{x\in\Omega_\sva}\sva^{15/8}\sigma_x\delta_x$ and $W^{\sva} = \sum_{x\in\Omega_\sva}\sva \omega_x^\sva\delta_x$. We will choose a function $f_{N,m}^\sva=f_{N,m}^\sva(W^{\sva},\widetilde{\Phi}^\sva)\geq 0$ that is
measurable with respect to $\Fc_{N,m}^\sva$, which can be interpreted as a tilting of the underlying measure. By the Cauchy-Schwarz inequality we have
\begin{equation}
\begin{aligned}
\Eb\Et\left[(\Qc_{N,m}^\sva)^{1/2}\right]
& =\Eb\Et\left[(f_{N,m}^\sva)^{1/2}(\Qc_{N,m}^\sva)^{1/2}(f_{N,m}^\sva)^{-1/2}\right] \\
& \leq \Eb\Et\left[f_{N,m}^\sva\Qc_{N, m}^\sva\right]^{1/2}\Eb\Et\left[(f_{N,m}^\sva)^{-1}\right]^{1/2}.
\end{aligned}
\end{equation}
Recalling the definition of $\Qc_{N,m}^\sva$ and that $f_{N, m}^\sva$ is $\Fc_{N, m}^\sva$ measurable, we have that
\begin{align*}
\Eb\Et\left[f_{N, m}^\sva\Qc_{N, m}^\sva\right]
 = \Eb\Et\left[f_{N, m}^\sva\Eb\Et\left[\frac{\d \mu^{\omega,\sva}}{\d \mu^{\sva}}\Big|\Fc^\sva_{N,m}\right]\right]
 = \Eb\Et\left[f_{N, m}^\sva\frac{\d \mu^{\omega,\sva}}{\d \mu^{\sva}}\right].
\end{align*}
Applying Fubini's theorem, it therefore suffices to show that
\begin{equation}\label{fraclim}
\lim_{N\rightarrow \infty}\lim_{m\to\infty} \lim_{\sva\downarrow 0}\Et\left[\Eb\left[f_{N, m}^\sva\frac{\d \mu^{\omega,\sva}}{\d \mu^{\sva}}\right]\right]^{1/2}\Et\left[\Eb\left[(f_{N, m}^\sva)^{-1}\right]\right]^{1/2}=0.
\end{equation}
To bound $\Eb\big[f_{N, m}^\sva\frac{\d \mu^{\omega,\sva}}{\d \mu^{\sva}}\big]$, we will interpret $\frac{\d \mu^{\omega,\sva}}{\d \mu^{\sva}}$ as a change of measure for $\omega^\sva_\cdot$ as follows. Write
$$
W^\sva_\lambda:=\sum_{x\in\Omega_\sva}\sva\lambda(x)\omega_x^\sva\delta_x
$$
and, with a slight abuse of notation, $\langle\widetilde{\Phi}^\sva,W^\sva_\lambda\rangle = \sum_{x\in\Omega_\sva}\lambda_x^\sva\sigma_x\omega_x^\sva= \sum_{x\in\Omega_\sva}\sva^{7/8} \lambda(x)\sigma_x\omega_x^\sva
$. Recall that $\theta_\sva=e^{-\frac{1}{2}\sva^{-1/4}||\lambda||^2_{L^2}}$. By definition we have
\begin{align*}
\frac{\d \mu^{\omega,\sva}}{\d \mu^{\sva}}(\widetilde{\Phi}^\sva)
 = \frac{\theta_\sva}{\widetilde{Z}^{\omega,\sva}}\exp\Bigg(\sum_{x\in\Omega_\sva}\lambda_x^\sva\sigma_x\omega_x^\sva\Bigg)
 = \frac{\theta_\sva \exp\big(\langle\widetilde{\Phi}^\sva,W^\sva_\lambda\rangle\big)}{\widetilde{Z}^{\omega,\sva}}.
\end{align*}
We now consider the change of measure $\widetilde{\Pb}$ on the disorder defined by
\begin{equation}\label{e:CoM}
\frac{\d \widetilde{\Pb}}{\d \Pb}(\omega^\sva_\cdot)=\frac{\theta_\sva\exp\big(\langle \widetilde{\Phi}^\sva,W^\sva_\lambda\rangle\big)}{\Eb\Big[\theta_\sva\exp\big(\langle\widetilde{\Phi}^\sva,W^\sva_\lambda\rangle\big)\Big]},
\end{equation}
where the denominator tends to $1$ as $a\downarrow 0$. This change of measure shifts the disorder $(\omega^\sva_x)$ into a field of independent variables with unit variance but spin-dependent centring as is stated in Lemma \ref{l:oDist}. We omit the proof which follows immediately by manipulating the Laplace transforms.
\begin{lem}\label{l:oDist}
Under $\widetilde{\Pb}$ for $\sigma$ fixed, the family $\{\omega_x^\sva\}_{x\in\Omega_\sva}$ are independent Gaussian random variables with means $\kappa_x:=\lambda_x^\sva\sigma_x$ and unit variance.
\end{lem}

To prove \eqref{fraclim}, it suffices to find a suitable choice of $f_{N, m}^\sva$ such that
\begin{equation}\label{f-1}
\lim_{N\to\infty} \lim_{m\to\infty} \limsup_{\sva\downarrow 0}\Et\big[\Eb\big[(f_{N, m}^\sva)^{-1}\big]\big] <\infty,
\end{equation}
and
\begin{equation}\label{fNa}
\lim_{N\rightarrow \infty} \lim_{m\to\infty} \lim_{\sva\downarrow 0}\Et\left[\widetilde{\Eb}\left[f_{N, m}^\sva/\widetilde{Z}^{\omega,\sva}\right]\right]=0.
\end{equation}
For this, we need to choose $f_{N, m}^\sva$ to be typically close to $1$ but small on the rare events that $\frac{\d \mu^{\omega,\sva}}{\d \mu^{\sva}}$ is large.  We choose
\begin{equation}\label{fNadef}
f_{N, m}^\sva:=\exp\left(-S_N^\sva\ind_{\{X_{N,m}^\sva\geq M_N^\sva\}}\right)
\end{equation}
where $M_N^\sva, S_N^\sva$ are constant sequences to be chosen later, which do not depend on $m$ and diverge as $\sva\rightarrow0$ then $N\rightarrow\infty$, and $X_{N,m}^\sva=X_{N,m}^\sva(W^{\sva},\widetilde{\Phi}^\sva)$ is an approximation of $\langle\widetilde{\Phi}^\sva,W^\sva_\lambda\rangle$ which is $\Fc_{N, m}^\sva$ measurable. By doing this, we have the desired behaviour that $f_{N, m}^\sva$ is typically close to $1$ with respect to $\Pb$, but small when $\langle\widetilde{\Phi}^\sva,W^\sva_\lambda\rangle$ is large, which occurs with high probability with respect to $\widetilde{\Pb}$.

We wish to approximate the spin configuration $\sigma$ by random variables measurable with respect to $\Fc_{N,m}^\sva$. More precisely, we will replace $\sigma_x$ for $x\in B_{i,j}^{N,\sva}$ by the constant value
\begin{equation}\label{rhoN}
\rho_{i,j}^{N,\sva}:=\begin{cases}1, & \text{if } \sum_{x\in B_{i,j}^{N,\sva}} (\lambda^\sva_x)^2 \sigma_x  \geq0,\\
-1, & \text{if } \sum_{x\in B_{i,j}^{N,\sva}} (\lambda^\sva_x)^2 \sigma_x<0.\end{cases}
\end{equation}
Note that $(\rho_{i,j}^{N,\sva})^2=1$ which does not depend on $\widetilde{\Phi}^\sva$. We then approximate $\langle \widetilde{\Phi}^\sva,W^\sva_\lambda\rangle$ by
\begin{align}\label{XNa}
X_{N, m}^\sva:= \sva^{-1/8}\sum_{i,j=1}^N\rho_{i,j}^{N,\sva} W^{\sva,N,m}_{\lambda, i,j}
= \sva^{-1/8}\sum_{i,j=1}^N\rho_{i,j}^{N,\sva} 2^{-m} \Big\lfloor 2^m \sum_{x\in B_{i,j}^{N,\sva}}\sva \lambda_x^\sva \omega_x^\sva \Big\rfloor,
\end{align}
where $W^{\sva, N, m}_{\lambda, i, j}$ is defined as in \eqref{WPNmij}. Note that
\begin{equation}\label{XNa2}
\begin{gathered}
\lim_{m\to\infty} X_{N,m}^\sva  = X_{N}^\sva:= \sva^{-1/8}\sum_{i,j=1}^N\rho_{i,j}^{N,\sva} W^{\sva,N}_{\lambda, i,j}
= \sum_{i,j=1}^N\rho_{i,j}^{N,\sva} \sum_{x\in B_{i,j}^{N,\sva}} \lambda^\sva_x \omega_x^\sva, \\
\mbox{and} \qquad |X_{N,m}^\sva - X_{N}^\sva| \leq 2^{-m}N^2 \sva^{-1/8}.
\end{gathered}
\end{equation}

We are now ready to prove \eqref{f-1} for suitable choices of $S_N^\sva$ and $M_N^\sva$.
\begin{lem}\label{l:gBnd}
Suppose that $\lim\limits_{N\rightarrow \infty}\lim\limits_{\sva\downarrow 0}S_N^\sva=\infty$ and set $M_N^\sva:=\left(4S_N^\sva\sum_{x\in\widehat\Omega_\sva}(\lambda_x^\sva)^2\right)^{1/2}$. Then
\[\lim_{N\rightarrow \infty}\lim_{m\to\infty} \lim_{\sva\downarrow 0}\Eb\left[(f_{N, m}^\sva)^{-1}\right]<\infty.\]
\end{lem}
\begin{proof}
By the definition of $f_{N, m}^\sva$ and the bound in \eqref{XNa2}, we have that
\begin{align*}
\Eb\left[(f_{N, m}^\sva)^{-1}\right]
&=1+(e^{S_N^\sva}-1)\Pb(X_{N,m}^\sva>M_N^\sva)\\
&\leq 1+(e^{S_N^\sva}-1)\Pb\big(X_{N}^\sva>M_N^\sva - 2^{-m}N^2 \sva^{-1/8}\big).
\end{align*}
From the definition of $X^\sva_N$ in \eqref{XNa2} and our choice of $\omega^\sva_x$, we note that $X_N^\sva$ is a centred Gaussian random variable with variance
\[
\sum_{i,j=1}^N\Big(\rho_{i,j}^{N,\sva}\Big)^2\sum_{x\in B_{i,j}^{N,\sva}}(\lambda_x^\sva)^2=\sum_{x\in\widehat\Omega_\sva}(\lambda_x^\sva)^2
= \sva^{-1/4} (1+o(1)) \Vert \lambda\Vert_{L^2(\widehat\Omega)}^2.
\]
It then follows by a Gaussian tail estimate that
\begin{align*}
&\Pb\big(X_N^\sva>M_N^\sva- 2^{-m}N^2 \sva^{-1/8}\big)\\
\leq\, & \exp\left(-\frac{(M_N^\sva - 2^{-m}N^2 \sva^{-1/8})^2}{2\sum_{x\in\widehat\Omega_\sva}(\lambda_x^\sva)^2}\right)\cdot \frac{\left(\sum_{x\in\widehat\Omega_\sva}(\lambda_x^\sva)^2\right)^{1/2}}{(M_N^\sva - 2^{-m}N^2 \sva^{-1/8}) \sqrt{2\pi}}.
\end{align*}
In particular, by our choice of $M_N^\sva$ we have the desired result.
\end{proof}

We are now ready to prove \eqref{fNa} and thus conclude the proof of \eqref{e:Null}.
\begin{lem}\label{l:gConv}
There exists $S_N^\sva$ satisfying $\lim_{N\rightarrow \infty}\lim_{\sva\downarrow 0}S_N^\sva=\infty$ such that
\[\lim_{N\rightarrow \infty}\lim_{m\to\infty} \lim_{\sva\rightarrow 0}\Et\left[\widetilde{\Eb}\left[f_{N, m}^\sva/\widetilde{Z}^{\omega,\sva}\right]\right]=0\]
where in the definition of $f_{N, m}^\sva$ in \eqref{fNadef}, $M_N^\sva:=\left(4S_N^\sva\sum_{x\in\widehat\Omega_\sva}(\lambda_x^\sva)^2\right)^{1/2}$.
\end{lem}
\begin{proof}
First note that for $\varepsilon>0$
\begin{align}\label{fdZ}
\widetilde{\Eb}\left[f_{N, m}^\sva/\widetilde{Z}^{\omega,\sva}\right]
\leq \frac{1}{\varepsilon}\widetilde{\Eb}\left[f_{N, m}^\sva\ind_{\widetilde{Z}^{\omega,\sva}\geq \varepsilon}\right]
+ \widetilde{\Eb}\left[\ind_{\widetilde{Z}^{\omega,\sva}<\varepsilon}/\widetilde{Z}^{\omega,\sva}\right].
\end{align}
By Fubini's theorem and the change of measure \eqref{e:CoM}, we have that
\begin{align*}
\Et\left[\widetilde{\Eb}\left[\ind_{\widetilde{Z}^{\omega,\sva}<\varepsilon}/\widetilde{Z}^{\omega,\sva}\right]\right]
 = (1+o(1))\Pb(\widetilde{Z}^{\omega,\sva}<\varepsilon),
\end{align*}
which can be made arbitrarily small uniformly in $\sva>0$ by choosing $\varepsilon>0$ sufficiently small since the weak limit of $\widetilde Z^{\omega, \sva}$, $\Zc^W$, does not have an atom at $0$ by Lemma \ref{l:mom}.

We then have that
\[
\widetilde{\Eb}\left[f_{N, m}^\sva\ind_{\widetilde{Z}^{\omega,\sva}\geq \varepsilon}\right]\leq \widetilde{\Eb}\left[f_{N, m}^\sva\right]=1+(e^{-S_N^\sva}-1)\widetilde{\Pb}(X_{N,m}^\sva>M_N^\sva) \leq \widetilde{\Pb}(X_{N,m}^\sva\leq M_N^\sva) + e^{-S_N^\sva}.
\]
Since $|X^\sva_{N,m}-X^\sva_N|\leq 2^{-m}N^2 \sva^{-1/8} =:C_{N,m}\sva^{-1/8}$ by \eqref{XNa2}, to conclude the proof, it suffices to show that
$$
\lim_{N\rightarrow \infty}\lim_{m\to\infty}\lim_{\sva\downarrow 0}\Et\left[\widetilde{\Pb}(X_{N}^\sva\leq M_N^\sva+C_{N,m}\sva^{-1/8})\right]=0.
$$

By Lemma \ref{l:oDist}, $\omega_x^\sva$ are independent Gaussian random variables with mean $\lambda_x^\sva\sigma_x$ and unit variance with respect to $\widetilde{\Pb}$. It follows that $X_N^\sva$ is a Gaussian random variable with mean and variance given by
\begin{align*}
m_{N,\sva} = \sum_{i,j=1}^N\rho_{i,j}^{N,\sva}\sum_{x\in B_{i,j}^{N,\sva}}(\lambda_x^\sva)^2\sigma_x, \qquad
s_{N,\sva}^2 = \sum_{x\in \widehat\Omega_\sva}(\lambda_x^\sva)^2 = \sva^{-1/4} (1+o(1)) \Vert \lambda\Vert_{L^2(\widehat\Omega)}^2
\end{align*}
respectively where we have used that $(\rho^{N, \sva}_{i,j})^2=1$.

If $M_N^\sva + C_{N,m}\sva^{-1/8} < m_{N,\sva}$, then by a Gaussian tail estimate,  we have that
\[\widetilde{\Pb}(X_N^\sva\leq M_N^\sva +C_{N,m}\sva^{-1/8})\leq \frac{\exp\left(-\frac{1}{2}\left(\frac{m_{N,\sva}-M_N^\sva-C_{N,m}\sva^{-1/8}}{s_{N,\sva}}\right)^2\right)}{\sqrt{2\pi}\frac{m_{N,\sva}-M_N^\sva-C_{N,m}\sva^{-1/8}}{s_{N,\sva}}}.\]
It therefore suffices to show that $(m_{N,\sva}-M_N^\sva-C_{N,m}\sva^{-1/8})/s_{N,\sva}$ tends to $\infty$ in $\Pt$-probability as $\sva\downarrow 0$, then $m\to\infty$ followed by $N\rightarrow \infty$.

Recalling the definitions of $M_N^\sva$ and $s_{N,\sva}$ and that $C_{N,m}=2^{-m}N^2$, we have that $M_N^\sva/s_{N,\sva}=(4S_N^\sva)^{1/2}$ and $\lim_{m\to\infty}\lim_{\sva\downarrow 0}C_{N,m}\sva^{-1/8}/s_{N, \sva}=0$.  Therefore, since we can choose $S_N^\sva$ tending to $\infty$ as slowly as we like, it suffices to show that $m_{N,\sva}/s_{N,\sva}$ tends to $\infty$, or equivalently,
\[
\sva^{\frac18} m_{N, \sva} = \sva^{\frac18}\sum_{i,j=1}^N\rho_{i,j}^{N,\sva}\sum_{x\in B_{i,j}^{N,\sva}}(\lambda_x^\sva)^2\sigma_x= \sum_{i,j=1}^N\Bigg|\sum_{x\in B_{i,j}^{N,\sva}}\sva^{\frac{15}{8}}\lambda(x)^2\sigma_x\Bigg| \to \infty
\]
as $\sva\downarrow 0, \, m\to\infty,\, N\to\infty$
where we have used the definition that $\rho^{N, \sva}_{i,j}$ is the sign of $\sum_{x\in B_{i,j}^{N,\sva}}(\lambda_x^\sva)^2\sigma_x$.
This divergence follows from Lemma \ref{l:IConv} below.
\end{proof}

\begin{lem}\label{l:IConv}
In $\Pt^\sva$-distribution,
\[\lim_{N\rightarrow \infty}\lim_{\sva\rightarrow 0}\sum_{i,j=1}^N\Bigg|\sum_{x\in B_{i,j}^{N,\sva}}\sva^{15/8}\lambda(x)^2\sigma_x\Bigg|= \infty.\]
\end{lem}
\begin{proof}
We first note that the sum we are interested in is equal to
\[\sum_{i,j=1}^NN^{-15/8}\Bigg|\sum_{x\in B_{i,j}^{N,\sva}}(\sva N)^{15/8}\lambda(x)^2\sigma_x\Bigg|
\ed \sum_{i,j=1}^NN^{-15/8}\Bigg|\sum_{x\in \widetilde{B}_{i,j}^{N,\sva}}\sva_N^{15/8}\lambda_N(x)^2\widetilde{\sigma}_x\Bigg|\]
where we have rescaled space such that $\widetilde{B}_{i,j}^{N,\sva}$ is the discretisation of the box $[i-1,i]\times[j-1,j]$ with mesh size $\sva_N=\sva N$, $\lambda_N(x)=\lambda(x/N)$ and $\widetilde{\sigma}$ is an Ising configuration on $\cup  \widetilde{B}_{i,j}^{N,\sva}$. Write
\begin{equation}
Y_{i,j}^{N,\sva}:=\Bigg|\sum_{x\in\widetilde{B}_{i,j}^{N,\sva}} \sva_N^{15/8}\lambda_N(x)^2\widetilde{\sigma}_x\Bigg|
\end{equation}
then we will show that $(Y_{i,j}^{N,\sva})$ stochastically dominates a family which obeys a law of large numbers.

Write $D_{i,j}^{N,\sva}:=[i-2,i+1]\times[j-2,j+1]$ for the box of side length $3$ with $\widetilde{B}_{i,j}^{N,\sva}$ at the centre. We can choose a set $\red{I_N}\subset \{(i,j):1\leq i,j\leq N\}$ of size at least $(N-2)^2/9$ such that the boxes $(D_{i,j}^{N,\sva})_{(i,j)\in I^N}$ are contained within $[0,N]^2$ and do not intersect each other. This gives a collection of disjoint annuli $A_{i,j}^{N,\sva}:=D_{i,j}^{N,\sva}\setminus\widetilde{B}_{i,j}^{N,\sva}$ with side length $3$ and the centre box of side length $1$ removed.

We wish to show that, with high probability, a positive proportion of these annuli contain a $+$ Ising loop. Let
\begin{equation}
\Ac_{i,j}^N:=\{(i,j)\in I_N: \partial\widetilde{B}_{i,j}^{N,\sva} \stackrel{-}{\nleftrightarrow} \partial D_{i,j}^{N,\sva} \}
\end{equation}
denote the event that $\widetilde{B}_{i,j}^{N,\sva}$ is surrounded by a $+$ loop in the annulus $A_{i,j}^{N,\sva}$. A consequence of the RSW result \cite{duhono11} (see also \cite[(3.3)]{cagane16}) is that, for $\sva>0$ suitably small, there exists a constant $p>0$ independent of the boundary condition $\xi$ on $D_{i,j}^{N,\sva}$ such that
\begin{equation*}
\Pt^{\sva,\xi}_{D_{i,j}^{N,\sva}}(\Ac_{i,j}^N)>p.
\end{equation*}
The number of disjoint annuli that contain such a loop stochastically dominates a binomial random variable with $(N-2)^2/9$ trials and success probability $p$. In particular, with high probability, there are at least $cN^2$ such loops for some $c>0$ and all $N$ large. Write $\widetilde{I}_N\subset I_N$ to be the collection with such a loop and $\gamma_{i,j}$ the loop corresponding to $(i,j)$ where we choose any such loop if there are more then one.

%

Conditioned on these loops, the random variables $Y_{i,j}^{N,\sva}$ for $(i,j)\in\widetilde{I}_N$ are independent and positive. It therefore suffices to show that $\Pt^{\sva}(Y_{i,j}^{N,\sva}>c_1|(i,j)\in\widetilde{I}_N,\gamma_{i,j})>c_2$ for some $c_1,c_2>0$ which are independent of the pair $(i,j)$ and the loop $\gamma_{i,j}$. By the Paley-Zygmund inequality
\begin{align*}
&\Pt^\sva\left(Y_{i,j}^{N,\sva}>\frac{1}{2}\Et^\sva\left[Y_{i,j}^{N,\sva}\big|(i,j)\in\widetilde{I}_N,\gamma_{i,j}\right]  \Big|  (i,j)\in\widetilde{I}_N,\gamma_{i,j}\right)\\
\geq\, &
\frac{\Et^\sva\left[Y_{i,j}^{N,\sva}\big|(i,j)\in\widetilde{I}_N,\gamma_{i,j}\right]^2}{4\Et^\sva\left[(Y_{i,j}^{N,\sva})^2\big|(i,j)\in\widetilde{I}_N,\gamma_{i,j}\right]},
\end{align*}
therefore it suffices to show that independently of the pair $(i,j)$ and the loop $\gamma_{i,j}$ we have that $\Et[Y_{i,j}^{N,\sva}|(i,j)\in\widetilde{I}_N,\gamma_{i,j}]$ is bounded below and $\Et^\sva[(Y_{i,j}^{N,\sva})^2|(i,j)\in\widetilde{I}_N,\gamma_{i,j}]$ is bounded above.

For the lower bound, by the FKG inequality, we have that
\begin{align*}
\Et^\sva\left[Y_{i,j}^{N,\sva}\big|(i,j)\in\widetilde{I}_N,\gamma_{i,j}\right]
&\geq \sum_{x\in \widetilde{B}_{i,j}^{N,\sva}}\sva_N^{15/8}\lambda_N(x)^2\Et^\sva\left[\widetilde{\sigma}_x\big|(i,j)\in\widetilde{I}_N,\gamma_{i,j}\right]\\
&\geq  \inf_{x\in\Omega}\lambda(x)^2\sum_{x\in \widetilde{B}_{i,j}^{N,\sva}}\sva^{15/8}_N\Et^{\sva,+}_{D_{i,j}^{N,\sva}}\left[\widetilde{\sigma}_x\right]
\end{align*}
which is bounded below by the convergence result \cite[Theorem 1.2]{chhoiz15}.

For the upper bound we have that
\begin{align*}
\Et^\sva\left[(Y_{i,j}^{N,\sva})^2  \Big|  (i,j)\in\widetilde{I}_N,\gamma_{i,j}\right]
&=\sum_{x,y\in\widetilde{B}_{i,j}^{N,\sva}}\sva_N^{15/4}\lambda_N(x)^2\lambda_N(y)^2\Et^\sva\left[\widetilde{\sigma}_x\widetilde{\sigma}_y \Big|  (i,j)\in\widetilde{I}_N,\gamma_{i,j}\right] \\
&\leq  \Vert\lambda\Vert^4_{\infty}\sum_{x,y\in \widetilde{B}_{i,j}^{N,\sva}}\sva_N^{15/4}\Et^{\sva,+}_{\widetilde{B}_{i,j}^{N,\sva}}\left[\widetilde{\sigma}_x\widetilde{\sigma}_y\right]
\end{align*}
which is bounded above uniformly in $\sva_N$ by \cite[Proposition 3.9]{cagane15} and \cite[Theorem 1.1]{chhoiz15}.
\end{proof}

\begin{appendix}
\section{Convergence of Radon-Nikodym derivatives} \label{s:Radon}

In this appendix, we prove an elementary measure theoretic result needed in the proof.

\begin{lem}\label{l:Radon}
Let $(\mu_n)_{n\in\Nb}$ and $(\nu_n)_{n\in\Nb}$ be two sequences of probability measures on a complete separable metric space $B$ equipped with Borel $\sigma$-algebra $\Bc$. Suppose that $\mu_n\Rightarrow \mu$ and $\nu_n\Rightarrow \nu$, and furthermore, $\nu_n$ is absolutely continuous with respect to $\mu_n$ with Radon-Nikodym derivative $f_n(x)$ such that $\{f_n \in L^1(B, \Bc, \mu_n)\}_{n\in\Nb}$ are uniformly integrable. Then $\nu$ is absolutely continuous with respect to $\mu$.
\end{lem}
\begin{proof} Let $(X_n)_{n\in\Nb}$ be random variables defined on a probability space $(\Gamma, \Fc, \Pb)$ such that $X_n$  has law $\mu_n$. Denote $Y_n=f_n(X_n)$. Then $\Eb[Y_n]=1$, and the sequence of random vectors $(X_n, Y_n)_{n\in\Nb}$ is tight and admits convergent subsequences. Let us restrict to such a subsequence and still denote it by $(X_n, Y_n)_{n\in\Nb}$ for notational simplicity. Furthermore, by Skorokhod's representation theorem, we can couple $(X_n, Y_n)_{n\in\Nb}$ and its limit $(X, Y)$ with marginal laws $\mu$ and $\nu$ respectively, such that $(X_n, Y_n)\to (X, Y)$ almost surely on the probability space $(\Gamma, \Fc, \Pb)$.

For any bounded continuous function $g$ on $(B, \Bc)$, we then have
$$
\int g(x) \nu_n({\rm d}x) = \int g(x) f_n(x) \mu_n({\rm d}x) = \Eb[g(X_n) Y_n].
$$
The left hand side converges to $\int g(x) \nu({\rm d}x)$, while the right hand side converges to $\Eb[g(X) Y]$ by the uniform integrability assumption on $(Y_n)_{n\in\Nb}$. Since $X$  has law $\mu$, it follows that $\nu$ is absolutely continuous with respect to $\mu$ with Radon-Nikodym derivative $f(X)= \Eb[Y | X]$.
\end{proof}

\begin{rem}\label{r:Radon}
In Lemma \ref{l:Radon}, we cannot deduce how singular $\nu$ is with respect to $\mu$ from an assumption such as $\sup_n \int f_n^\alpha {\rm d}\mu_n\leq \varepsilon$ for some $\alpha \in (0,1)$ and $\varepsilon$ small, because from the proof of Lemma \ref{l:Radon}, we obtain
$$
\Eb[Y^\alpha] \leq \liminf_{n\to\infty} \Eb[Y_n^\alpha] = \int f_n^\alpha {\rm d}\mu_n \leq \varepsilon,
$$
but we cannot draw this conclusion if $Y$ is replaced by $f(X)=\Eb[Y|X]$, the Radon-Nikodym derivative of $\nu$ with respect to $\mu$. Indeed, if $\mu$ is a delta measure, then so would be $\nu$, and we would always have $\int f^\alpha {\rm d}\mu=1$ regardless of the value of $\alpha$ and $\varepsilon$. However, if $(\mu_n)_{n\in\Nb}$ and $\mu$ are atomic measures supported on the same discrete set containing no cluster points, then indeed,
\begin{equation}\label{falpha}
\int f^\alpha {\rm d}\mu \leq \liminf_{n\to\infty} \int f_n^\alpha {\rm d}\mu_n, \qquad \alpha \in (0,1).
\end{equation}
\red{Therefore, in order to prove the singularity in Theorem \ref{t:Sing}, it is not sufficient to show that
\[
\lim_{N\to\infty}\liminf_{\sva\downarrow 0}\Eb\Et[(\Qc_N^\sva)^{1/2}]=0 \qquad (\mbox{recall } \Qc_N^\sva \mbox{ from } \eqref{QNa}).
\]
Instead, we need to discretise the random variables as in \eqref{WPNmij} and then prove their weak convergence in Lemma \ref{l:aprox},
so that \eqref{falpha} can be applied.}
\end{rem}

\section{Positive moments of the partition functions} \label{s:mom}

In this appendix, we give uniform bounds on positive moments of the partition function $\widetilde Z^{\omega, \sva}_{\Omega; \lambda, h}$ as stated in Lemma \ref{l:mom}. The proof is based on hyper-contractivity. Similar estimates have been obtained in \cite{casuzy19} for the directed polymer model in dimension $2+1$.

We first formulate it explicitly as a lemma.
\begin{lem}\label{l:posmom}
Let $(\widetilde{Z}^{\omega,\sva}_{\Omega;\lambda,h})_{\sva\in (0,1]}$ be the RFIM partition function with $+$ boundary condition as in Theorem \ref{t:JntCnv}. Then for any $p>0$,
\begin{equation}
\limsup_{\sva\downarrow 0} \Eb\Big[\Big(\widetilde{Z}^{\omega,\sva}_{\Omega;\lambda,h}\Big)^p\Big] <\infty.
\end{equation}
\end{lem}
\begin{proof} We first recall from \eqref{e:Prefactor} the following expansion, where $\xi^\sva_x:=\lambda_x^\sva\omega^\sva_x+h_x^\sva$:
\begin{flalign}
\widetilde{Z}_{\Omega;\lambda, h}^{\omega,\sva}
& = \theta_\sva\cosh(\xi^\sva_\cdot)^{\Omega_\sva}\sum_{I\subseteq\Omega_\sva}\Et_{\Omega}^\sva\left[\sigma_\cdot^I\right] \tanh(\xi^\sva_\cdot)^I.
\end{flalign}
By Lemma \ref{l:Theta}, the prefactor $\theta_\sva\cosh(\xi_\cdot^\sva)^{\Omega_\sva}$ converges to $1$ in probability and is uniformly bounded in $L^p$ for any $p>0$. Therefore we will drop this prefactor from now on.

By Taylor expansion and the assumption that $\omega$ has finite exponential moments, it is easily seen that
\begin{gather*}
\mu^\sva_x:=\Eb[\tanh \xi_x^\sva]  = h_x^\sva+O(\sva^{21/8}), \qquad (\vartheta^\sva_x)^2:= \mathbb{V}{\rm ar}[\tanh \xi_x^\sva]  = (\lambda_x^\sva)^2+O(\sva^{7/2}), \\
\Eb[(\tanh \xi_x^\sva)^{2k}]  = (\lambda_x^\sva)^{2k}+O(\sva^{7k/2}) \quad \forall\, k\in\Nb.
\end{gather*}
Denoting $\eta^\sva_x:= (\xi^\sva_x-\mu^\sva_x)/\vartheta^\sva_x$, then we can write
$$
\Psi_\sva(\eta^\sva_\cdot) := \sum_{I\subseteq\Omega_\sva}\Et_{\Omega}^\sva\left[\sigma_\cdot^I\right]
\prod_{x\in I} (\vartheta^\sva_x \eta^\sva_x + \mu^\sva_x) =: \sum_{J\subset \Omega_\sva} \psi_\sva(J) \prod_{x\in J} \eta^\sva_x,
$$
which is a polynomial chaos expansion with kernel $\psi_\sva$ in the family of independent random variables $(\eta^\sva_x)_{x\in\Omega_\sva}$, which has mean $0$, variance $1$, and uniformly bounded $2k$-th moment for each $k\in\Nb$. Therefore we can apply hypercontractivity to bound its moments \red{(see e.g.~\cite[Theorem B.1]{casuzy19} for the exact form needed here, or the original reference \cite[Proposition 3.11]{moodol10} for a proof).} More specifically,  for any $p>2$, we can find $c_p\in (1, \infty)$ such that
\begin{align}
\Eb\big[|\Psi(\eta^\sva_\cdot)|^p\big]^{\frac2p} \leq \sum_{J\subset\Omega_\sva} c_p^{2|J|} \psi_a(J)^2
& = \Eb\Big[ \Big(\sum_{J\subset \Omega_\sva} \psi_\sva(J) \prod_{x\in J} c_p\eta^\sva_x\Big)^2\Big] \notag\\
& = \Eb\big[\Psi_\sva(c_p\eta^\sva_\cdot)^2\big] = \Eb[\Psi_\sva(c_p\eta^\sva_\cdot)]^2 + \mathbb{V}{\rm ar}(\Psi_\sva(c_p\eta^\sva_\cdot)). \label{Psicpa}
\end{align}
Note that
$$
\Eb[\Psi_\sva(c_p\eta^\sva_\cdot)] = \Eb[\Psi_\sva(\eta^\sva_\cdot)] = (1+o(1)) \Eb[\widetilde{Z}_{\Omega;\lambda, h}^{\omega,\sva}]
= (1+o(1)) \Et^\sva_\Omega\big[e^{\sum_x \sigma_x h^\sva_x}\big] \to \Et_\Omega[e^{\langle \Phi, h\rangle}] 
$$
which is finite by \cite[Corollary 3.8]{cagane15} on the convergence of the exponential moment of the Ising magnetisation field. Therefore to bound
$\Eb\big[|\Psi(\eta^\sva_\cdot)|^p\big]$, it only remains to bound $\mathbb{V}{\rm ar}(\Psi_\sva(c_p\eta^\sva_\cdot))$.

We can write
$$
\Psi_\sva(c_p\eta^\sva_\cdot) = \sum_{I\subseteq\Omega_\sva}\Et_{\Omega}^\sva\left[\sigma_\cdot^I\right] \prod_{x\in I} (c_p\vartheta^\sva_x \eta^\sva_x + \mu^\sva_x) = \sum_{I\subseteq\Omega_\sva}\Et_{\Omega}^\sva\left[\sigma_\cdot^I\right] \prod_{x\in I} c_p\vartheta^\sva_x (\eta^\sva_x + \tilde \mu^\sva_x),
$$
where $\tilde \mu^\sva_x := \mu^\sva_x/c_p\vartheta^\sva_x = (\sva+o(\sva)) h(x)/c_p\lambda(x)$. By \cite[Lemma 4.1]{casuzy17},
in order to bound $\mathbb{V}{\rm ar}(\Psi_\sva(c_p\eta^\sva_\cdot))$, we can remove $\tilde \mu^\sva_x$ from $\eta^\sva_x+\tilde\mu^\sva_x$, and for any $\varepsilon>0$, bound
\begin{align*}
\mathbb{V}{\rm ar}(\Psi_\sva(c_p\eta^\sva_\cdot)) & \leq e^{\varepsilon^{-1}\sum_{x\in \Omega_\sva} (\tilde\mu^\sva_x)^2}
\sum_{I \subset\Omega_\sva} (1+\varepsilon)^{|I|}\Et_{\Omega}^\sva\left[\sigma_\cdot^I\right]^2 \prod_{x\in I} (c_p\vartheta^\sva_x)^2 \\
& \leq e^{\varepsilon^{-1}(c_p^{-2}+o(1)) \Vert h/\lambda\Vert_{L^2(\Omega)}^2} \sum_{I \subset\Omega_\sva} \big((1+\varepsilon)c_p^2|\lambda|_\infty^2\big)^{|I|} f^2_\Omega(I) \sva^{2k} \\
& \leq C \sum_{k=0}^\infty \big((1+\varepsilon)c_p^2|\lambda|_\infty^2\big)^{k} \frac{\Vert f_\Omega\Vert_{L^2(\Omega^k)}^2}{k!}\\
& \leq C \sum_{k=0}^\infty \big((1+\varepsilon)c_p^2|\lambda|_\infty^2\big)^{k} \frac{C^k(k!)^{1/4}}{k!} <\infty,
\end{align*}
where we have used the bounds \eqref{e:fUpp} and \eqref{fOmbd}. This shows that the right hand side of \eqref{Psicpa} is uniformly bounded for $\sva$ small, which concludes the proof of Lemma \ref{l:mom}.
\end{proof}

\section{Negative moments of the partition functions} \label{s:pos}

In this appendix, we give uniform bounds on negative moments of the partition function $\widetilde Z^{\omega, \sva}_{\Omega; \lambda, h}$ as stated in Lemma \ref{l:mom}, which implies that the continuum RFIM partition function $\Zc^W_{\Omega; \lambda, h}$ has negative moments of all orders; in particular, it is almost surely strictly positive. The proof is based on a concentration of measure bound for convex and locally Lipschitz functions of i.i.d.\ random variables, which is applied to $\log \widetilde  Z^{\omega, \sva}_{\Omega; \lambda, h}$ and then passed to the limit as $\sva\downarrow 0$. Similar estimates have been obtained for the directed polymer model~\cite{mo14} and the disordered pinning model~\cite{catoto17}. We will follow closely \cite{casuzy19} where such calculations were done for the directed polymer in dimension $2+1$, using the inequality \eqref{eq:thebound} proved in \cite{catoto17}.

\begin{lem}\label{l:negmom}
Let $(\widetilde{Z}^{\omega,\sva}_{\Omega;\lambda,h})_{\sva\in (0,1]}$ be the RFIM partition function with $+$ boundary condition as in Theorem \ref{t:JntCnv}. Assume that an i.i.d.\ sequence $(\omega_n)_{n\in\Nb}$ with the same law as $\omega_x$ satisfies the following concentration of measure inequality:
\begin{gather}
	\exists\, \gamma>1, C_1, C_2 \in (0,\infty) \text{\ such that for all $n\in\Nb$ and } \red{g}: \Rb^n \to \Rb
	\text{ convex and $1$-Lipschitz}, \notag \\
	\Pb \Big( \big| \red{g}(\omega_1, \ldots, \omega_N) - M_g \big|
	\ge t\Big) \le C_1e^{-C_2 t^\gamma} \,, \quad M_g \mbox{ is a median of } \red{g}(\omega_1, \ldots, \omega_N). \label{assD2}
\end{gather}
Then there exists $c=c(\lambda, h)\in (0,\infty)$ such that for all $\sva \in (0,1)$ sufficiently small, we have
\begin{align} \label{eq:boundf}
	\forall\, t \ge 0: \qquad
	\ \Pb(\log \widetilde  Z^{\omega, \sva}_{\Omega; \lambda, h}\leq -t) \leq c \, e^{- t^\gamma / c}.
\end{align}
In particular, $\limsup_{a\in (0,1]} \Eb[(\widetilde  Z^{\omega, \sva}_{\Omega; \lambda, h})^{-p}]<\infty$ for all $p\in (0,\infty)$.
\end{lem}
\begin{rem}
We can improve \eqref{eq:boundf} to a uniform bound with respect to the boundary condition $\xi\in \{\pm1\}^{\Omega_\sva}$ if we had a uniform upper bound on $\Eb[(\widetilde Z^{\omega, \sva, \xi}_{\Omega;\lambda, h})^2]$, which unfortunately is missing because the spin correlation functions obtained in \cite{chhoiz15} is only for special boundary conditions and cannot be used to dominate spin correlation functions for general boundary conditions.
\end{rem}

Condition \eqref{assD2} is satisfied if $\omega$ is either bounded, or Gaussian, or has a density of the form $\exp(-V(\cdot) + U(\cdot))$ where $V$ is uniformly strictly convex and $U$ is bounded. See \cite{le01} for more details. In particular, by choosing $(\omega_x)_{x\in \Zb^2}$ to be i.i.d.\ standard Gaussian and applying Theorem~\ref{t:JntCnv} and Fatou's lemma, we obtain the following.

\begin{cly}\label{c:pos}
Let $\Zc^W_{\Omega; \lambda, h}$ be the continuum RFIM partition function with $+$ boundary condition as in Theorem \ref{t:JntCnv}. Then $\Eb[(\Zc^W_{\Omega; \lambda, h})^{-p}]<\infty$ for all $p\in (0,\infty)$, and $\Zc^W_{\Omega; \lambda, h}>0$  $\Pb$-a.s.
\end{cly}

\begin{proof}[Proof of Lemma \ref{l:negmom}]
\red{The first fact we need is that for $(\omega_n)_{n\in\mathbb N}$ that satisfy \eqref{assD2}, one can also obtain a lower deviation bound for $f(\omega_1, \ldots, \omega_n)$ where $f$ is only convex and locally Lipschitz. More precisely, by \cite[Proposition~3.4]{catoto17}, for every $n\in\Nb$ and every convex differentiable function $f\colon \Rb^n\to \Rb$,}
the following bound holds for all $b\in \Rb$ and $t,c\in (0,\infty)$,
\begin{align} \label{eq:thebound}
	\Pb\big(f(\omega)\leq b-t \big) \,
	\,\Pb\big(f(\omega)\geq b, |\nabla f(\omega) | \leq c \big)
	\leq \red{C_1} \exp\Big( -\frac{(t/c)^\gamma}{\red{C_2}} \,\Big),
\end{align}
where $\omega = (\omega_1, \ldots, \omega_n)$ and
$|\nabla f(\omega)|:=\sqrt{\sum_{i=1}^n(\partial_i f(\omega))^2}$ is the norm of the gradient.
We can then deduce \eqref{eq:boundf} by applying \eqref{eq:thebound} to $f_\sva(\omega) = \log \widetilde  Z^{\omega, \sva}_{\Omega; \lambda, h}$.

It suffices to show that for some $b$ to be chosen, there exist $c=c(\lambda, h) \in (0,\infty)$
and $\theta \in (0,1)$ such that
\begin{align}\label{concentrate1}
	\liminf_{a\downarrow 0} 	\Pb \big( f_\sva(\omega) \geq b \,,
	|\nabla f_\sva(\omega)|\leq c \big) >0 \,.
\end{align}
First note that for any $c > 0$, we have
\begin{align} \label{eq:deba}
	\Pb\big( f_\sva(\omega)\geq b \,,|\nabla f_\sva(\omega)|\leq c\big) = \Pb\big(f_\sva(\omega)\geq b\big) -
	\Pb\big(f_\sva(\omega)\geq b \,,|\nabla f_\sva(\omega)|> c\big).
\end{align}
Note that by the definition of $\widetilde Z^{\omega, \sva}_{\Omega; \lambda, h}$ in \eqref{tildeZ}, Lemma \ref{l:Theta}, and \cite[Corollary 3.8]{cagane15} on the convergence of the exponential moment of the Ising magnetisation field, we have
$$
\lim_{\sva\downarrow 0}\Eb [\widetilde Z^{\omega, \sva}_{\Omega; \lambda, h}]=\lim_{\sva\downarrow 0} (1+o(1))\Et^\sva_\Omega[e^{\sum \sigma_x h^\sva_x}] = \Et_\Omega[e^{\langle \Phi, h\rangle}] =:C_{1,h}>0.
$$
On the other hand, we have $\lim_{\sva\downarrow 0}\Eb [ (\widetilde Z^{\omega, \sva}_{\Omega; \lambda, h})^2 ] = \Eb[ (\Zc^W_{\Omega; \lambda, h})^2] =:C_{2, \lambda h}<\infty$ by \cite[Theorems 3.14 \& 2.3]{casuzy17}.

Choosing $b= \log (C_{1,h}/2)$, we can bound the first probability in \eqref{eq:deba} by Paley-Zygmund inequality:
\begin{align}\label{PL}
	\Pb\big(f_\sva(\omega)\geq b) = \Pb\big( \widetilde Z^{\omega, \sva}_{\Omega; \lambda, h} \geq C_{1,h}/2\big) =  	
	\Pb\Big( \widetilde Z^{\omega, \sva}_{\Omega; \lambda, h} \geq \tfrac{1+o(1)}{2} \Eb [\widetilde Z^{\omega, \sva}_{\Omega; \lambda, h}] \Big)
	\geq \frac{C_{1,h}^2}{5C_{2, \lambda, h}}>0.
\end{align}

To bound the second term in \eqref{eq:deba}, we first compute that for each $x\in \Omega_\sva$,
\begin{align*}
\frac{\partial f_\sva(\omega)}{\partial \omega_{x}} = \frac{1}{Z^{\omega, \sva}_{\Omega; \lambda, h}}
\Et_{\Omega} \Big[ \sigma_x \lambda^\sva_x e^{\sum_x \sigma_x(\lambda^\sva_x\omega^\sva_x+h^\sva_x) }] =
\Et^{\omega, \sva}_{\Omega;\lambda, h}[\lambda^\sva_x \sigma_x]
\end{align*}
\begin{align*}
\mbox{and} \qquad |\nabla f_\sva(\omega)|^2 = \sum_{x\in \Omega_\sva}
\Big(\frac{\partial f_\sva}{\partial \omega_{x}}\Big)^2 =
\Et^{\omega, \sva, \otimes 2}_{\Omega;\lambda, h} [L_\lambda(\sigma, \sigma')],
\end{align*}
where $L_\lambda(\sigma, \sigma'):= \sum_{x\in \Omega_\sva} (\lambda^\sva_x)^2 \sigma_x \sigma_x'\geq 0$ for two spin configurations $\sigma$ and $\sigma'$ sampled independently according to $\Pt_{\Omega}^\sva$, which is the overlap between $\sigma$ and $\sigma'$ weighted by $(\lambda^\sva_\cdot)^2$.

On the event that $f_\sva(\omega)\geq b$, that is $\widetilde Z^{\omega, \sva}_{\Omega; \lambda, h}\geq C_{1,h}/2$, we can bound
\begin{align*}
	|\nabla f_\sva(\omega)|^2 \leq \frac{4\theta_a^2}{C_{1,h}^2} \,
	                                             \Et^{\sva, \otimes2}_{\Omega}\Big[L_\lambda(\sigma, \sigma') \exp\Big\{\sum_x(\sigma_x+\sigma_x')(\lambda^\sva_x \omega^\sva_x + h^\sva_x) \Big\} \Big].
\end{align*}
Therefore we have
\begin{align} \label{nabfa}
	&\Pb\big(f_\sva(\omega)\geq b\,,|\nabla f_\sva(\omega)|> c\big)\\
\leq \ &  \frac{1}{c^2} \,\,\Eb \Big[ |\nabla f_\sva(\omega)|^2 \,
	\ind_{\{ f_\sva(\omega)\geq b \}} \Big] \notag\\
\leq \ & \frac{4 \theta_\sva^2}{c^2 C_{1,h}^2} \Eb\Big[ \Et^{\sva, \otimes2}_{\Omega}\Big[L_\lambda(\sigma, \sigma') \exp\Big\{\sum_x(\sigma_x+\sigma_x')(\lambda^\sva_x \omega^\sva_x + h^\sva_x) \Big\} \Big] \Big] \notag \\
= \ & \frac{4}{c^2 C_{1,h}^2} \Big(\theta_\sva^2 e^{\sum_x (\lambda^\sva_x)^2}\Big)
\Et^{\sva, \otimes2}_{\Omega}\Big[L_\lambda(\sigma, \sigma') e^{L_\lambda(\sigma, \sigma') + \sum_x \sigma_x h^\sva_x + \sum_x \sigma_x' h^\sva_x }\Big], \notag
\end{align}
where we note that $\theta_\sva^2 e^{\sum_x (\lambda^\sva_x)^2} \to 1$ as $\sva\downarrow 0$ by the same argument as in \red{Lemma \ref{l:Theta}}, while the expectation can be bounded by applying \red{Cauchy}-Schwarz and noting that
$$
\Et^{\sva, \otimes2}_{\Omega}\Big[e^{2\sum_x \sigma_x h^\sva_x + 2\sum_x \sigma_x' h^\sva_x }\Big] = \Et^{\sva}_{\Omega}\Big[e^{2\sum_x \sigma_x h^\sva_x} \Big]^2 \to \Et_\Omega\big[e^{\langle \Phi, 2h\rangle}\big]^2 <\infty
$$
by \cite[Corollary 3.8]{cagane15}, and
$$
\begin{aligned}
\Et^{\sva, \otimes2}_{\Omega}\Big[L^2_\lambda(\sigma, \sigma') e^{2L_\lambda(\sigma, \sigma')}\Big] \leq \Et^{\sva, \otimes2}_{\Omega}\big[e^{4L_\lambda(\sigma, \sigma')}\big] & = e^{-4\sum_x (\lambda^\sva_x)^2} \Eb\Big[\Et_{\Omega}^\sva\Big[e^{\sum_x 2\lambda^\sva_x \omega_x \sigma_x}\Big]^2\Big] \\
& = e^{-4\sum_x (\lambda^\sva_x)^2} e^{\sva^{-1/4}\Vert 2\lambda\Vert _{L^2}^2}  \Eb\big[(\widetilde Z^{\omega, \sva}_{\Omega; 2\lambda, 0})^2\big],
\end{aligned}
$$
where the product of the two exponentials converges to 1 by the same argument as in \red{Lemma \ref{l:Theta}}, and $\Eb\big[(\widetilde Z^{\omega, \sva}_{\Omega; 2\lambda, 0})^2\big]\to C_{2, 2\lambda,\red{0}} <\infty$ by \cite[Theorems 3.14 \& 2.3]{casuzy17}.

Choosing $c$ large enough, we can make the right hand side of \eqref{nabfa} sufficiently small such that \eqref{concentrate1} holds. The bound \eqref{eq:boundf} then follows from the concentration inequality \eqref{eq:thebound}.
\end{proof}

\section{Besov-H\"older distributions integrated over subdomains} \label{s:BH}

In this appendix, we show that for any $\Phi \in \Cc^{\alpha}_{loc}(\Omega)$ with $\alpha \in (-1,0)$, it is possible to define the integral of $\Phi$ over a subdomain $B\subset \Omega$ with a regular enough boundary, which can be applied in particular to the continuum RFIM magnetisation field. This is used in the proof of Theorem \ref{t:Sing} in Section \ref{s:Sing}.

\begin{lem}\label{l:test}
Let $\Omega$ be open, $\alpha \in (-1,0)$, and $\bar{B}\subset \Omega$ with $\partial B$ having upper box dimension $d<2+\alpha$, which is defined by
$$
d=\limsup_{\epsilon\downarrow 0} \frac{\log N(\epsilon)}{\log (1/\epsilon)}
$$
where $N(\epsilon)$ is the minimal number of boxes of side length $\epsilon$ needed to cover $\partial B$. Then the map $\langle \cdot,\ind_B\rangle:\Cc^\alpha_{loc}(\Omega)\rightarrow \Rb$ is well defined and continuous.
\end{lem}
\begin{proof}
We first consider the space $\Cc^\alpha$ instead of $\Cc^\alpha_{loc}(\Omega)$. By definition, any $f\in\Cc^\alpha$ is the limit of a sequence $(f_m)_{m\geq 1}\subset C_c^\infty$ in $\Cc^\alpha$. We then define
$$
\langle f,\ind_B\rangle:=\lim_{m\rightarrow \infty}\langle f_m,\ind_B\rangle.
$$
We need to show that this limit exists and is unique. In particular,  if another sequence $(\tilde{f}_m)_{m\geq 1}\subset C_c^\infty$ converges to $f$ in $\Cc^\alpha$, then $\langle f_m,\ind_B\rangle-\langle \tilde{f}_m,\ind_B\rangle=\langle f_m-\tilde{f}_m,\ind_B\rangle$
converges to $0$ as $m\to\infty$. It then suffices to show that for any $(f_m)_{m\geq 1}\subset C_c^\infty$ converging to $0$ in $\Cc^\alpha$, we have $\langle f_m,\ind_B\rangle\to 0$. The continuity of $\langle \cdot, \ind_B\rangle$ would also follow as a consequence.

First choose $\bar d>d$ and $n_0\in\Nb$ such that for all $n\geq n_0$,
$$
\frac{\log N(2^{-n})}{\log 2^n} \leq \bar d <2+\alpha.
$$
In other words, on the lattice $\Lambda_n=\Zb^2/2^n$, we need at most $2^{n\bar d+2}$ squares of side length $2^{-n}$ to cover $\partial B$.
Since $f_m\in L^2$, by the multi-resolution analysis from \eqref{e:VW} and \eqref{e:fL2}, we have that
\begin{align}
\langle f_m,\ind_B\rangle
&= \left\langle \Vs_{n_0}f_m+\sum_{n=n_0}^\infty\Ws_nf_m,\ind_B\right\rangle\notag\\
&= \sum_{x\in\Lambda_{n_0}}\langle f_m,\phi_{n_0,x}\rangle_{L^2}\langle\phi_{n_0,x},\ind_B\rangle_{L^2} + \sum_{n=n_0}^\infty\sum_{x\in\Lambda_n, i=1,2,3}\langle f_m,\psi^{(i)}_{n,x}\rangle_{L^2}\langle\psi^{(i)}_{n,x},\ind_B\rangle_{L^2}.\label{e:decom}
\end{align}
Since $B$ is bounded and $\phi_{n_0,x}$ is supported on $B(x,2^{-n_0})$, there are only finitely many $x\in\Lambda_{n_0}$ such that $\langle\phi_{n_0,x},\ind_B\rangle_{L^2}\neq 0$. Since for each $x$, $\Vert f_m\Vert_{\Cc^\alpha}\rightarrow 0$ implies that $\langle f_m,\phi_{n_0,x}\rangle_{L^2}\to 0$, we have that the first term in \eqref{e:decom} converges to $0$ as $m\rightarrow \infty$.

Similarly, since each $\psi^{(i)}_{n,x}$ is supported on $B(x,2^{-n})$ and $\int \psi^{(i)}_{n,x}(y){\rm d}y=0$, only $x\in \Lambda_n=\Zb^2/2^n$ within distance $2^{-n}$ from $\partial B$ can have $\langle\psi^{(i)}_{n,x},\ind_B\rangle_{L^2}\neq 0$. The number of such $x$ is bounded by $2^{n\bar d+2}$, and for each such $x$, we can estimate more quantitatively
$$
|\langle\psi^{(i)}_{n,x},\ind_B\rangle_{L^2}| = \Big|\int_B 2^n \psi^{(i)}(2^n(y-x)) {\rm d}y \Big| \leq  2^{-n} \int_{2^n(B-x)} |\psi^{(i)}(z)| {\rm d}z
\leq C 2^{-n}
$$
uniformly in $x$ and $n$. On the other hand, by the definition of $\Vert f\Vert_{\Cc^\alpha}$ in \eqref{Calpha1},
\begin{align*}
\sup_{x\in\Lambda_n}|\langle f_m,\psi^{(i)}_{n,x}\rangle_{L^2}|
&= 2^{-(1+\alpha)n} \sup_{x\in\Lambda_n} \Big|2^{(2+\alpha)n}\int f_m(y) \psi^{(i)}(2^n(y-x)){\rm d}y \Big|\\
&\leq 2^{-(1+\alpha)n}\Vert f_m\Vert_{\Cc^\alpha}.
\end{align*}
Therefore, using the assumption that $\alpha>-1$ and $\bar d<2+\alpha$, the second term in \eqref{e:decom} is bounded above by
\begin{equation}
C\sum_{n=0}^\infty\sum_{i=1,2,3} 2^{n(\bar d-1)}\cdot 2^{-(1+\alpha)n} \Vert f_m\Vert_{\Cc^\alpha}
\leq C' \Vert f_m\Vert_{\Cc^\alpha},
\end{equation}
which converges to $0$ as $m\rightarrow \infty$. This concludes the proof that $\langle \cdot,\ind_B\rangle:\Cc^\alpha(\Omega)\rightarrow \Rb$ is well defined and continuous. The proof for the case $\Cc^\alpha_{loc}(\Omega)$ is similar, where we replace $f\in \Cc^\alpha_{loc}(\Omega)$ by $f \chi \in \Cc^\alpha$ with $\chi \in C_c^\infty(\Omega)$ and $\chi \ind_B=\chi$.
\end{proof}

\end{appendix}

\section*{Acknowledgements}
We wish to thank in particular F.~Caravenna for fruitful discussions during various stages of this project, which helped us to overcome some of the technical difficulties. We would also like to thank C.~Garban, Y.~Le Jan, and N.~Zygouras for helpful discussions, and an anonymous referee for suggesting a simplified proof of tightness. R.~Sun and A.~Bowditch are supported by NUS grants R-146-000-260-114 and R-146-000-300-114.


\bibliographystyle{imsart-number} 
\bibliography{ref.bib}

\end{document}